\documentclass[format=acmsmall, review=false, anonymous=False]{acmart}

\usepackage{wrapfig}
\usepackage{bbm}
\usepackage{dsfont}
\usepackage{tikz}
\usepackage{amsmath,amsfonts,bbm,natbib,commath,mathtools} 

\usepackage{booktabs} 
\usepackage{caption}
\usepackage{subfig}
\usepackage[shortlabels]{enumitem}
\usepackage{algorithm}
\usepackage[noend]{algpseudocode}

\newcommand{\lineplanningproblem}{Real-Time Line Planning Problem}
\newcommand{\lpp}{\textsc{Rlpp}}
\newcommand{\platform}{platform}

\newcommand{\tcar}{t^{\text{car}}}
\newcommand{\bx}{\mathbf{x}}
\newcommand{\by}{\mathbf{y}}

\usepackage{color-edits}
\addauthor{ch}{blue}
\addauthor{sb}{magenta}
\addauthor{ss}{orange}
\addauthor{np}{red}

\usepackage[capitalize]{cleveref}

\usepackage{hyperref}
\hypersetup{colorlinks=true,linkcolor=magenta,citecolor=magenta}

\begin{document}~
\title{Real-Time Approximate Routing for Smart Transit Systems}
    \author{Siddhartha Banerjee}
    \email{sbanerjee@cornell.edu}
    \affiliation{%
    \institution{Cornell University}}
    \author{Chamsi Hssaine}
    \email{ch822@cornell.edu}
    \affiliation{%
    \institution{Cornell University}}
	\author{No\'emie P\'erivier}
    \email{np2708@columbia.edu}
    \affiliation{%
    \institution{Columbia University}}
    \author{Samitha Samaranayake}
    \email{samitha@cornell.edu}
    \affiliation{%
    \institution{Cornell University}}

\begin{abstract}
We study real-time routing policies in smart transit systems, where the platform has a combination of cars and high-capacity vehicles (e.g., buses or shuttles) and seeks to serve a set of incoming trip requests. The platform can use its fleet of cars as a feeder to connect passengers to its high-capacity fleet, which operates on fixed routes. Our goal is to find the optimal set of (bus) routes and corresponding frequencies to maximize the social welfare of the system in a given time window. This generalizes the {\it Line Planning Problem}, a widely studied topic in the transportation literature, for which existing solutions are either heuristic (with no performance guarantees), or require extensive computation time (and hence are impractical for real-time use). To this end, we develop a $1-\frac1e-\varepsilon$ approximation algorithm for the {\it Real-Time Line Planning Problem}, using ideas from randomized rounding and the Generalized Assignment Problem. Our guarantee holds under two assumptions: $(i)$ no inter-bus transfers and $(ii)$ access to a pre-specified set of feasible bus lines. We moreover show that these two assumptions are crucial by proving that, if either assumption is relaxed, the \lineplanningproblem{} does not admit any constant-factor approximation. Finally, we demonstrate the practicality of our algorithm via numerical experiments on real-world and synthetic datasets, in which we show that, given a fixed time budget, our algorithm outperforms Integer Linear Programming-based exact methods.
\end{abstract}
\maketitle

\section{Introduction}

{In the past decade, the advent of ride-hailing platforms such as Lyft and Uber has revolutionized urban mobility. While commuter transit needs in cities were traditionally satisfied by personal vehicles or mass transit systems, ride-hailing platforms have grown immensely in popularity and gained a seemingly permanent footing in the landscape of mobility solutions. 
However, despite the increasingly important role played by \emph{Mobility-on-Demand (MoD)} services in today's society, the intermingling of various modes of transportation has yet to make its way into the status quo: by and large, if not for using their personal vehicles, commuters either choose to complete their trips in a low-capacity ride-hailing vehicle, or opt for public mass transit options, each of these options equipped with their respective benefits and disadvantages. On the one hand, ride-hailing services have been lauded for their convenience, competitive pricing, and the creation of flexible, gig economy jobs. On the other, these services have been associated with negative environmental impacts, chief of which are increased emissions due to higher volumes of traffic congestion and vehicle-miles traveled. Moreover, despite the fact that these options are often less expensive than taxi services, they remain out of reach for lower-income populations, for whom mass transit such as bus and subway services remains the most accessible option. And, while these public transit systems are more affordable and environmentally sustainable, they fail to adequately serve areas that are not as densely populated. Further, due to their inability to dynamically adapt to passenger demand, public transit vehicles are often overly packed during rush hour and significantly underfilled in off-peak hours~\citep{nyc_packed_subways}, an inefficiency from which ride-hailing options do not suffer.}

In light of this, it should be clear that there exist potentially massive gains from integrating the on-demand capabilities of ride-hailing services with mass transit options to create a smarter transportation system.  The benefits of such a synergy have been uncovered in both the academic literature~\citep{benefit_integration}, as well as in the wild, with ride-hailing platforms such as Lyft experimenting with mass transit-like options in recent years~\citep{hawkins2017lyft}.  Indeed, the need for such integration has become all the more stark throughout 2020, when cities have turned to microtransit as a means of addressing reduced public transit services due to the coronavirus pandemic~\citep{pandemic_microtransit}. The value of real-time, adaptive hybrid transportation options that retain both the convenience of ride-hailing and the sustainability of mass transit, is perhaps best evidenced by New York City's months-long overnight, for-hire vehicle program for essential workers, discontinued in August 2020 due to high costs~\citep{nyc_pandemic_buses}. {The extremes of the mobility spectrum to which the Metropolitan Transit Authority (MTA) turned as a stopgap in this relatively short period of time typifies the potential perils of relying on an unintegrated system: the free, late-night for-hire vehicle program was a boon to essential workers who had been deprived of a means to get to their shifts, but the city could not sustain this as a long-term solution; mass transit solutions, though sustainable, were not flexible enough to appropriately serve workers living in communities historically underserved by these services~\citep{nyc_underserved_comms}. As an alternative to these two extremes, the city recently turned to the creation of overnight bus routes that mirror workers' most popular trips~\citep{nyc_pandemic_buses}. In doing so, the MTA is faced with a number of fundamental questions upon which the success of such a system hinges: {given these essential workers' origins and destinations, {\it which} routes should the transit agency operate? {\it How frequently} should it operate each route? How can {\it short, for-hire vehicle trips} help to connect passengers to these routes?}} This paper aims to answer these questions in order to effectively operate such an integrated system.

Just as cities have yet to successfully operate integrated mobility services, the operations research and transportation communities have by and large studied ride-hailing and mass transit systems separately. On the one hand, there exists an active line of work on approximate-optimal policies for dispatching drivers to ride requests, and rebalancing empty vehicles~\citep{banerjee2016pricing,braverman2019empty,banerjee2018state,kanoria2019near}. 
On the other, the problem of designing the optimal bus routes to serve passenger demand
dates back to the mid-1970s~\citep{magnanti1984network}. And, though the question of integrating mass transit and single-occupancy vehicle solutions has attracted increasing attention in recent years, operational questions have largely been restricted to using ride-hailing services to connect to {\it pre-existing} transit networks~\citep{Ma,MA2019417}. The joint problem of adaptively designing bus routes in near real-time, and connecting passengers to these routes via ride-hailing services has to our knowledge yet to be explored.

The key obstacle in designing real-time algorithms with provable guarantees for transit-network design is the size of the decision space: the number of possible routes is exponential in the number of nodes of the road network. As such, approaches have either been heuristic~\citep{CEDER,Pape,Borndorfer} (lacking any guarantees), or exact~\citep{nachtigall2008simultaneous} (requiring extensive computation time); the former may lead to severe losses in efficiency, while the latter are more properly suited for designing {\it long-term} bus routes, rather than routes that adapt to changing demand patterns.

In this paper, we show that it is possible to design efficient algorithms for line planning that both provide passengers with the experience of near-real-time booking and service {\it and} have theoretical guarantees. 
However, this is only true up to a point: as the designer expands her solution space of feasible transit options, one runs into fundamental limits in terms of how good an approximation one can hope to achieve via efficient algorithms. Overall, our work provides \emph{theoretically sound and practically meaningful algorithms for real-time line planning, and also exposes the computational limits of line planning.}

\subsection{Summary of our contributions}

We consider a model in which a Mobility-on-Demand provider (henceforth \emph{platform}) has control of a vehicle fleet comprising both single-occupancy and high-capacity vehicles (henceforth \emph{cars} and \emph{buses} respectively). 
The platform is faced with a number of trip requests to fill during a window of time (e.g., one hour), and has full knowledge of passenger demands (source and destination locations, and constraints on start and end times) prior to the beginning of the time window. 
This assumption is practically motivated by scheduling services now offered by ride-hailing apps like Lyft and Uber, and/or the use of accurate demand forecasting models. 
The \platform{} can service these trip requests via different \emph{trip options}: it can send a car to transport the passenger from her source to her destination; it can use a car for the first and last legs of the passenger's trip, and have her travel by bus in between; or it can use more complicated trips comprising of multiple car and bus legs.

Each passenger matched to a trip option leads to an associated value (or \emph{reward}), which can reflect both the passenger's utility for the trip-time, comfort, transfers, etc., as well as platform costs in terms of car-miles; in addition, the \platform{} also incurs a cost for operating each bus route at a given frequency. We define the combination of a route and a frequency to be a {\it line}. The goal of the platform is to determine the optimal set of lines to operate (given a fixed budget $B$ for opening lines), as well as the assignment of passengers to trip options utilizing these lines, in order to maximize the total reward. 
We refer to this problem as the {\it \lineplanningproblem} (\lpp).
 
As discussed earlier, though there exist exact methods for solving the Line Planning Problem that can be adapted to the \lpp\ setting (e.g., by formulating and solving an associated integer linear program),
the extensive computation time required to obtain the optimal set of lines runs counter to our goal of computing short-term lines that adapt to demand patterns throughout the day. 
This motivates studying the task of finding good approximate solutions to \lpp. In this context, we make two contributions:
\begin{itemize}[leftmargin=*]
\item[1.] We first demonstrate the computational limits of \lpp\ by showing that no constant-factor approximation is possible if we relax any one of two assumptions: $(i)$ access to a pre-specified set of feasible bus lines, and $(ii)$ no inter-line (i.e., bus-to-bus) transfers.
\item[2.] Under both above assumptions, we design an efficient algorithm for \lpp\, that respects budget constraints with high probability, while guaranteeing a welfare that is within a $\left(1-\frac1e-\varepsilon\right)$-factor of the optimal (where $\varepsilon$ trades-off the quality of approximation and probability of exceeding the budget).
\end{itemize}
While assumptions $(i)$ and $(ii)$ are commonly made both in practice and in the academic literature, our work provides \emph{strong theoretical justifications} for these assumptions in that if either fails to hold, there is no hope of obtaining a constant-factor approximation.
Assumption $(i)$ forms the basis of all {exact} ILP-based methods; it is also practically relevant due to both constraints imposed by cities on bus routes, as well as expert knowledge of transit designers as to which routes are useful. Assumption $(ii)$ reflects a practical constraint that, given a passenger may already incur car-bus transfers  in the first/last legs of her trip, additional bus-bus transfers could be deemed excessive. 
Even when both hold, however, we show that the problem is still far from trivial: in particular, it does not inherit the attractive combinatorial property of submodularity, and so one cannot employ standard techniques to get the classical $1-\frac1e$ approximation guarantee~\citep{Wolsey}. 
Moreover, we also show that the natural linear programming (LP) relaxation has a worst-case integrality gap of at least $\frac12$. 

In spite of this, in our main technical contribution, we provide a $\left(1-\frac1e-\varepsilon\right)$-factor approximation for \lineplanningproblem{}. 
More specifically, our algorithm uses a novel LP relaxation followed by a randomized rounding procedure, that can be tuned to guarantee that the budget constraint is met with any desired high-probability bound, while losing an $\varepsilon$-factor in the welfare guarantee.
Our key technical insight is that the \lineplanningproblem{} can be relaxed and re-formulated as an exponential-size \emph{configuration LP}, and that this formulation then allows us to use ideas from randomized rounding for the Separable Assignment Problem~\citep{SAP}. We then leverage the additional structure in \lpp{} to show that the rounding step is the {\it only} source of loss in our algorithm. Our results hold under an assumption which we term {\it trip optimality} (i.e., of all the ways in which a passenger can join a given line via car, she must be assigned to the best such option). However, we later show how this assumption can be relaxed, and, with slight modification to our algorithm, we lose at most a constant factor. 

Finally, we investigate the practical efficacy of our approach via numerical experiments on real-world and synthetic datasets. We note that, although our algorithm does not guarantee a solution that is always within budget, in practice it is easy to run multiple replications (which are cheap, and can be run in parallel) and choose the best realization satisfying budget constraints. 
Our numerical experiments simulate this procedure, and we observe that given a time budget on computation (as would be necessary for real-time line planning), our algorithm outperforms state-of-the-art ILP solvers for large problem instances, thereby demonstrating its practicality for the problem of designing integrated and flexible transit networks at scale.

\noindent\textbf{Structure of the paper}.
In Section~\ref{sec:related}, we survey relevant literature. We present our model and define the \lineplanningproblem{} in Section~\ref{sec:preliminaries}. 
In Section~\ref{sec:hardness}, we characterize fundamental computational limits of \lpp, establishing the need for a candidate set of lines and precluding bus transfers; we also show that standard techniques are inadequate for our setting. We present our main algorithm and guarantees in Section~\ref{sec:main-result}, and back this up with numerical results in Section~\ref{sec:numerical-experiments} and Appendix~\ref{app:synthetic_experiments}. 
Extensions to our main results can be found in Appendix~\ref{ssec:extensions}. 

\section{Related work}
\label{sec:related}


\noindent\textbf{Line planning in public transportation:} {Our work falls under the large umbrella of transportation network design; see~\citet{magnanti1984network,guihaire2008transit,farahani2013review} for excellent expositions.} Much of this work has historically involved heuristics, including greedy approaches based on simpler network primitives such as shortest-paths and minimum spanning trees~\citep{Dubois,gattermann2017line}, and metaheuristics~\citep{zhao2006simulated,zhao2004transit}. 
The largest-scale use of heuristic methods is, to our knowledge, the work of~\citet{Borndorfer}, who rely on column generation and {greedy heuristics}; more importantly, the formulation requires allowing for \emph{arbitrarily many bus transfers}. 
In practice, it is desirable to enforce a maximum number of allowable transfers (something which we explicitly model in our work); enforcing this however severely impacts computational performance. 
In a followup work,~\citet{borndorfer2012direct} incorporate transfer penalties (a type of ``soft'' constraint), but the resulting algorithms require on the order of 10 hours of computation time, which for our setting is infeasible. 
More recently, exact methods based on ILP formulations have gained in popularity~\citep{wan2003mixed,barra2007solving,marin2009urban,nachtigall2008simultaneous}, though these only scale to small networks.

\noindent\textbf{Ride-pooling:} Our problem is also closely related to {\it ride-pooling}, where the goal is to combine multiple trips to improve the efficiency of ride-sharing platforms. To model trade-offs between passenger inconvenience and sharing rides,~\citet{Santi} introduced the abstraction of a {\it shareability network}, and showed via simulations that pairing up to two requests per vehicle could lead to significant savings in cumulative driver miles. Their methods, however, accommodate at most three passengers per vehicle (with heuristics). \citet{Alonso-Mora} develop algorithms which perform well (in simulations) for up to 10 passengers per vehicle. Their method is based on clique decompositions of the shareablity network, which again scales poorly with increasing vehicle capacity; it also imposes strict quality of service constraints leading to fewer feasible trip configurations, which may greatly reduce efficiency in the setting we consider.

\noindent\textbf{Multi-modal solutions to the first-mile/last-mile problem:} From a practical perspective, the transportation community has explored public-private partnerships to exploit both the high capacity of public transit buses and the flexibility of MoD fleets~\citep{benefit_integration,MA2019417}. These works, however, focus not on designing the transit network, but rather on dynamic vehicle dispatching and routing between origin or destination and transit hubs.

\noindent\textbf{Stochastic control for ride-sharing:} A more recent line of work has developed stochastic models for ride-sharing with trip requests arriving via a random process. This has enabled the use of techniques from stochastic control for scheduling and routing~\citep{banerjee2016pricing,braverman2019empty,banerjee2018state,kanoria2019near}, as well as the study of system-level questions such as the effect of competing platforms~\cite{sejourne2018price}.
The algorithms developed in these papers largely rely on assuming that under appropriate scaling (in particular, in the `large-market' scaling, where the number of cars scales with the demand), the system is well approximated via a steady-state problem. This is practically meaningful in ride-sharing systems, which can be thought of as being near-stationary over sufficiently small time-scales; such an assumption, however, critically depends on the impact of a single car being ``small'' relative to the rest of the system. In a setting with high-capacity vehicles, however, this ceases to be true, and it is unclear if a stochastic model of our system would exhibit the rapid mixing property with which low-capacity ride-sharing models are endowed, and which allows for these attractive guarantees.

\noindent\textbf{Randomized rounding for resource allocation problems:} Our methodological approach is inspired by the use of {\it configuration programs} for improved approximations for a number of  combinatorial optimization problems~\citep{MBA,LP_conf_scheduling,SAP}. At a high level, the approximation algorithms proposed in this line of work reformulate the resource allocation problem as an exponential-size integer program that optimizes over all feasible sets of resources; the LP relaxation of this program can be (approximately) solved in polynomial time, and used to produce approximately optimal solutions to the original problem via rounding. Our main result relies on the randomized rounding scheme proposed by~\citet{SAP} for the Separable Assignment Problem, which comprises a set of bins and items, with a separate packing constraint on each bin, and rewards for each item-bin pair. The objective is to pack items into bins such that the aggregate value of all packed items is maximized. The analogy to the \lineplanningproblem{} is natural: items correspond to passengers, bins correspond to lines, and the packing constraints correspond to capacity constraints for each bus. The key difference between these two problems is that, in the case of \textsc{Sap}, \emph{bins are provided in advance}, with no associated cost for using a bin. In contrast, the main difficulty in \lpp{} is in determining which lines to open, given costs for opening each line, and a budget constraint which further couples all lines (bins) together.

\section{Preliminaries}
\label{sec:preliminaries}

\subsection{System Model}

We model the transit network as an undirected weighted graph $G=(V,E)$, with $|V|=n$ potential origin/destination nodes, edges representing roads between these nodes, and edge weights $(\tau_e)_{e\in E}$ representing the \emph{cost} (for example, travel time) required to traverse an edge. 
We assume that $\tau_e \geq \tau_{\min}$ for some constant $\tau_{\min} > 0$. 

The network is operated by a single Mobility-on-Demand provider (henceforth \emph{platform}), which employs a fleet comprising two types of vehicles: single-occupancy vehicles (\emph{cars}), and high-capacity vehicles (\emph{buses}). The platform makes all scheduling and routing decisions in a centralized manner. These decisions are made over a fixed time-window, wherein prior to the beginning of the window, the platform receives a set of trip-requests (henceforth \emph{passengers}), and then must decide on a set of bus routes, and match passengers to these routes, using cars to cover `first-last mile' travel. The final trip option presented to each passenger must satisfy her travel needs, which we abstract via the notion of \emph{feasible trip options} for each passenger. The aim of the platform is to maximize some appropriate notion of \emph{system welfare}, which incorporates both utilities of passengers, and costs and constraints of the platform. 

\begin{figure}[!t]
\centering

\tikzset{every picture/.style={line width=0.75pt}} 

\begin{tikzpicture}[x=1.5pt,y=1.5pt,yscale=-1,xscale=1.1]

\draw [color={rgb, 255:red, 208; green, 2; blue, 27 }  ,draw opacity=1 ]   (140,170) -- (160,170) ;
\draw [shift={(160,170)}, rotate = 0] [color={rgb, 255:red, 208; green, 2; blue, 27 }  ,draw opacity=1 ][fill={rgb, 255:red, 208; green, 2; blue, 27 }  ,fill opacity=1 ][line width=0.75]      (0, 0) circle [x radius= 3.35, y radius= 3.35]   ;
\draw [shift={(140,170)}, rotate = 0] [color={rgb, 255:red, 208; green, 2; blue, 27 }  ,draw opacity=1 ][fill={rgb, 255:red, 208; green, 2; blue, 27 }  ,fill opacity=1 ][line width=0.75]      (0, 0) circle [x radius= 3.35, y radius= 3.35]   ;
\draw    (100,130) -- (100,150) ;
\draw    (100,150) -- (100,170) ;
\draw    (120,130) -- (120,150) ;
\draw    (120,150) -- (120,170) ;
\draw    (140,130) -- (140,150) ;
\draw [color={rgb, 255:red, 208; green, 2; blue, 27 }  ,draw opacity=1 ]   (140,150) -- (140,170) ;
\draw    (160,130) -- (160,150) ;
\draw    (180,150) -- (180,170) ;
\draw    (180,150) -- (180,170) ;
\draw    (160,150) -- (160,170) ;
\draw    (180,130) -- (180,150) ;
\draw [color={rgb, 255:red, 3; green, 35; blue, 255 }  ,draw opacity=1 ]   (101,123) -- (158,123) ;
\draw [shift={(161,123)}, rotate = 180] [fill={rgb, 255:red, 3; green, 35; blue, 255 }  ,fill opacity=1 ][line width=0.08]  [draw opacity=0] (5.36,-2.57) -- (0,0) -- (5.36,2.57) -- cycle    ;
\draw [color={rgb, 255:red, 0; green, 0; blue, 0 }  ,draw opacity=1 ]   (93,130) -- (93,140) -- (93,149) ;
\draw [shift={(93,152)}, rotate = 270] [fill={rgb, 255:red, 0; green, 0; blue, 0 }  ,fill opacity=1 ][line width=0.08]  [draw opacity=0] (5.36,-2.57) -- (0,0) -- (5.36,2.57) -- cycle    ;
\draw [color={rgb, 255:red, 0; green, 0; blue, 0 }  ,draw opacity=1 ] [dash pattern={on 4.5pt off 4.5pt}]  (100,156) -- (137,156) ;
\draw [shift={(140,156)}, rotate = 180] [fill={rgb, 255:red, 0; green, 0; blue, 0 }  ,fill opacity=1 ][line width=0.08]  [draw opacity=0] (5.36,-2.57) -- (0,0) -- (5.36,2.57) -- cycle    ;
\draw [color={rgb, 255:red, 0; green, 0; blue, 0 }  ,draw opacity=1 ]   (140,156) -- (157,156) ;
\draw [shift={(160,156)}, rotate = 180] [fill={rgb, 255:red, 0; green, 0; blue, 0 }  ,fill opacity=1 ][line width=0.08]  [draw opacity=0] (5.36,-2.57) -- (0,0) -- (5.36,2.57) -- cycle    ;
\draw [color={rgb, 255:red, 0; green, 0; blue, 0 }  ,draw opacity=1 ]   (166.5,150) -- (166.5,133) ;
\draw [shift={(166.5,130)}, rotate = 450] [fill={rgb, 255:red, 0; green, 0; blue, 0 }  ,fill opacity=1 ][line width=0.08]  [draw opacity=0] (5.36,-2.57) -- (0,0) -- (5.36,2.57) -- cycle    ;
\draw    (140,150) -- (160,150) ;
\draw    (160,150) -- (180,150) ;
\draw [shift={(180,150)}, rotate = 0] [color={rgb, 255:red, 0; green, 0; blue, 0 }  ][fill={rgb, 255:red, 0; green, 0; blue, 0 }  ][line width=0.75]      (0, 0) circle [x radius= 3.35, y radius= 3.35]   ;
\draw [shift={(160,150)}, rotate = 0] [color={rgb, 255:red, 0; green, 0; blue, 0 }  ][fill={rgb, 255:red, 0; green, 0; blue, 0 }  ][line width=0.75]      (0, 0) circle [x radius= 3.35, y radius= 3.35]   ;
\draw [color={rgb, 255:red, 208; green, 2; blue, 27 }  ,draw opacity=1 ]   (160,170) -- (180,170) 
;
\draw [shift={(180,170)}, rotate = 0] [color={rgb, 255:red, 208; green, 2; blue, 27 }  ,draw opacity=1 ][fill={rgb, 255:red, 208; green, 2; blue, 27 }  ,fill opacity=1 ][line width=0.75]      (0, 0) circle [x radius= 3.35, y radius= 3.35]   
;
\draw [shift={(160,170)}, rotate = 0] [color={rgb, 255:red, 208; green, 2; blue, 27 }  ,draw opacity=1 ][fill={rgb, 255:red, 208; green, 2; blue, 27 }  ,fill opacity=1 ][line width=0.75]      (0, 0) circle [x radius= 3.35, y radius= 3.35]   
;
\draw    (100,170) -- (120,170) ;
\draw [shift={(120,170)}, rotate = 0] [color={rgb, 255:red, 0; green, 0; blue, 0 }  ][fill={rgb, 255:red, 0; green, 0; blue, 0 }  ][line width=0.75]      (0, 0) circle [x radius= 3.35, y radius= 3.35]   ;
\draw [shift={(100,170)}, rotate = 0] [color={rgb, 255:red, 0; green, 0; blue, 0 }  ][fill={rgb, 255:red, 0; green, 0; blue, 0 }  ][line width=0.75]      (0, 0) circle [x radius= 3.35, y radius= 3.35]   ;
\draw    (120,170) -- (140,170) ;
\draw [shift={(140,170)}, rotate = 0] [color={rgb, 255:red, 0; green, 0; blue, 0 }  ][fill={rgb, 255:red, 0; green, 0; blue, 0 }  ][line width=0.75]      (0, 0) circle [x radius= 3.35, y radius= 3.35]   ;
\draw [shift={(120,170)}, rotate = 0] [color={rgb, 255:red, 0; green, 0; blue, 0 }  ][fill={rgb, 255:red, 0; green, 0; blue, 0 }  ][line width=0.75]      (0, 0) circle [x radius= 3.35, y radius= 3.35]   ;
\draw [color={rgb, 255:red, 0; green, 0; blue, 0 }  ,draw opacity=1 ]   (100,130) -- (120,130) 
;
\draw [shift={(120,130)}, rotate = 0] [color={rgb, 255:red, 0; green, 0; blue, 0 }  ,draw opacity=1 ][fill={rgb, 255:red, 0; green, 0; blue, 0 }  ,fill opacity=1 ][line width=0.75]      (0, 0) circle [x radius= 3.35, y radius= 3.35]   
;
\draw [shift={(100,130)}, rotate = 0] [color={rgb, 255:red, 0; green, 0; blue, 0 }  ,draw opacity=1 ][fill={rgb, 255:red, 0; green, 200; blue, 0 }  ,fill opacity=1 ][line width=0.75]      (0, 0) circle [x radius= 3.35, y radius= 3.35]   
;
\draw [color={rgb, 255:red, 0; green, 0; blue, 0 }  ,draw opacity=1 ]   (120,130) -- (140,130) ;
\draw [shift={(140,130)}, rotate = 0] [color={rgb, 255:red, 0; green, 0; blue, 0 }  ,draw opacity=1 ][fill={rgb, 255:red, 0; green, 0; blue, 0 }  ,fill opacity=1 ][line width=0.75]      (0, 0) circle [x radius= 3.35, y radius= 3.35]   
;
\draw [shift={(120,130)}, rotate = 0] [color={rgb, 255:red, 0; green, 0; blue, 0 }  ,draw opacity=1 ][fill={rgb, 255:red, 0; green, 0; blue, 0 }  ,fill opacity=1 ][line width=0.75]      (0, 0) circle [x radius= 3.35, y radius= 3.35]   
;
\draw [color={rgb, 255:red, 0; green, 0; blue, 0 }  ,draw opacity=1 ]   (140,130) -- (160,130) 
;
\draw [shift={(140,130)}, rotate = 0] [color={rgb, 255:red, 0; green, 0; blue, 0 }  ,draw opacity=1 ][fill={rgb, 255:red, 0; green, 0; blue, 0 }  ,fill opacity=1 ][line width=0.75]      (0, 0) circle [x radius= 3.35, y radius= 3.35]   
;
\draw    (160,130) -- (169.5,130) -- (180,130) ;
\draw [shift={(180,130)}, rotate = 0] [color={rgb, 255:red, 0; green, 0; blue, 0 }  ][fill={rgb, 255:red, 0; green, 0; blue, 0 }  ][line width=0.75]      (0, 0) circle [x radius= 3.35, y radius= 3.35]   ;
\draw [color={rgb, 255:red, 208; green, 2; blue, 27 }  ,draw opacity=1 ]   (100,150) -- (120,150) ;
\draw [shift={(120,150)}, rotate = 0] [color={rgb, 255:red, 208; green, 2; blue, 27 }  ,draw opacity=1 ][fill={rgb, 255:red, 208; green, 2; blue, 27 }  ,fill opacity=1 ][line width=0.75]      (0, 0) circle [x radius= 3.35, y radius= 3.35]   
;
\draw [shift={(100,150)}, rotate = 0] [color={rgb, 255:red, 208; green, 2; blue, 27 }  ,draw opacity=1 ][fill={rgb, 255:red, 208; green, 2; blue, 27 }  ,fill opacity=1 ][line width=0.75]      (0, 0) circle [x radius= 3.35, y radius= 3.35]   
;
\draw [color={rgb, 255:red, 208; green, 2; blue, 27 }  ,draw opacity=1 ]   (120,150) -- (140,150) ;
\draw [shift={(140,150)}, rotate = 0] [color={rgb, 255:red, 208; green, 2; blue, 27 }  ,draw opacity=1 ][fill={rgb, 255:red, 208; green, 2; blue, 27 }  ,fill opacity=1 ][line width=0.75]      (0, 0) circle [x radius= 3.35, y radius= 3.35]   
;
\draw [shift={(120,150)}, rotate = 0] [color={rgb, 255:red, 208; green, 2; blue, 27 }  ,draw opacity=1 ][fill={rgb, 255:red, 208; green, 2; blue, 27 }  ,fill opacity=1 ][line width=0.75]      (0, 0) circle [x radius= 3.35, y radius= 3.35]   
;
\draw [shift={(160,130)}, rotate = 0] [color={rgb, 255:red, 0; green, 0; blue, 0 }  ,draw opacity=1 ][fill={rgb, 255:red, 0; green, 200; blue, 0 }  ,fill opacity=1 ][line width=0.75]      (0, 0) circle [x radius= 3.35, y radius= 3.35]   
;
\draw (95,115.4) node [anchor=north west][inner sep=0.75pt]  [font=\normalsize]  {$s$};
\draw (164,115.4) node [anchor=north west][inner sep=0.75pt]  [font=\normalsize]  {$d$};
\draw (105,110) node [anchor=north west][inner sep=0.75pt]   [align=left] {{\small direct travel by car}};
\draw (52,135) node [anchor=north west][inner sep=0.75pt]   [align=left] {{\small hybrid travel}};
\draw (148,179) node [anchor=north west][inner sep=0.75pt]  [font=\small,color={rgb, 255:red, 208; green, 2; blue, 27 }  ,opacity=1 ] [align=left] {bus line};
\end{tikzpicture}
\caption{Example transit network with a single bus route (marked in red) and a single passenger traveling from source node $s$ to destination node $d$ (marked in green). The passenger has $2$ trip options: she can travel directly by car from $s$ to $d$ (blue arrow), or use a hybrid trip option comprising the dashed portion of the bus route, completing the rest of the trip by car (solid black arrow).}
\label{fig:passenger-options}
\end{figure}

\noindent\textbf{Vehicle Fleet Model.} {As mentioned above, the platform controls both a fleet of cars (which can serve a single passenger) and buses (which are high-capacity). Since in most ride-hailing systems, the former fleet is much larger, and has a high density throughout the city, we primarily focus on the routing/scheduling decisions for buses, incorporating the constraints and costs of the car fleet in the value function of passengers.
}

Buses have a fixed capacity $C \in \mathbb{N}$, corresponding to the maximum number of passengers a bus can simultaneously accommodate. 
We define a {\it route} $r$ to be a fixed sequence of consecutive edges of $G$, and let $\mathcal{R}$ denote the set of all routes of cost at most $D \in (0,T]$, where $D$ is a constant determined by the platform (for example, the duration of the longest bus ride such that the trip is completed within the time window). Moreover, the platform is said to serve route $r \in \mathcal{R}$ at frequency $f \in \mathbb{N}$ if $f$ buses traverse $r$ during the time window.
A key abstraction in this paper is that of a {\it line}, which we formally define below.
\begin{definition}[Line]
The platform is said to operate a {\it line} $\ell = (r_{\ell},f_{\ell})$ if it runs high-capacity vehicles on route $r_{\ell}$ at frequency $f_{\ell}$. 
\end{definition}
We use $\mathcal{L} = \left\{(r, f) \mid (r,f) \in \mathcal{R}\times\mathbb{N}\right\}$ to denote the set of all \emph{feasible} lines the platform can operate, and let $L=|\mathcal{L}|$. Note that a line can accommodate at most $C\times f_{\ell}$ passengers for each edge $e \in r_{\ell}$, and as such it is without loss of generality to assume that $f_{\ell} \in \{1,\ldots,\lceil N/C \rceil\} \, \forall \, \ell \in \mathcal{L}$, where $N$ is the total number of trip requests during the time window.

The platform has a \emph{budget} $B \in \mathbb{R}_+$ with which to open a set of lines. Let $c_{\ell}$ denote the cost of operating line $\ell$.
We assume that line costs are {{\it strictly increasing}} and {\it subadditive} in the frequencies. That is, suppose lines $\ell_1$ and $\ell_2$ use the same route $r$ and have frequencies $f_1, f_2$, respectively. Then:
\begin{enumerate}[$(i)$]
\item strictly increasing: $f_1 < f_2 \implies c_{\ell_1} < c_{\ell_2}$
\item subadditive: $c_{\ell_1} + c_{\ell_2} \leq c_{\ell_3}$, where $\ell_3 = (r, f_1 + f_2)$. 
\end{enumerate}

\noindent\textbf{Passenger Model.}
We use $\mathcal{P}$ to denote the set of all passengers requesting a trip during the time window, and $N = |\mathcal{P}|$ the total number of all such passengers. Each passenger $p\in \mathcal{P}$ is associated with fixed source and destination nodes $(s_p,d_p)$. 
To travel between these nodes, she can use a combination of cars and buses: in particular, she can travel directly from $s_p$ to $d_p$ exclusively by car; alternatively, she can travel by bus for the `middle leg' of her journey, and use cars for the first and last legs (if source/destination is not on the bus route). Figure~\ref{fig:passenger-options} illustrates these possibilities.

In principle, a more complex trip option could also involve multiple bus segments. In this work, however, we restrict passengers to take one of the above two trip options. 

\textbf{Assumption 1 (No inter-bus transfers).}
A trip can only comprise of a {\it single} bus leg; i.e., the platform cannot assign any passenger to multiple lines. 

From a practical perspective, this is a reasonable assumption, given that a passenger may already incur two transfers for the first and last miles of her trip. 
More importantly, in Section~\ref{sec:hardness} we show that if we relax this assumption by allowing the platform to use trip options involving even just two inter-line transfers, then we can not hope to achieve any constant-factor approximation.

Given line $\ell$, let $\Omega_{\ell p}$ denote the set of all trip options matching passenger $p$ to line $\ell$ that are \emph{feasible}, i.e., where the passenger completes her journey within the time window. Formally, 
\begin{align*}
\Omega_{\ell p} = \Big\{(s_p,i,j,d_p)|\; i,j \in r_{\ell}, 
p \text{ travels } s_p \rightarrow i \text{ and } j \rightarrow d_p \text{ by car, and } i \rightarrow j \text{ by bus line } \ell\Big\}
\end{align*}
Let {$\Omega_p = \{(\omega, \ell):\omega\in\Omega_{\ell p}, \ell \in \mathcal{L}\}$.} 
For each passenger $p$, there is an associated reward (or value) function {$v_p: \Omega_{p} \mapsto \mathbb{R}_+$}, representing the quality (from either the platform or the passenger's perspective) of a trip option using line $\ell$ (including potential costs incurred by the platform for the passenger's short car trip). We assume that $v_{p}(\cdot)$ is {\it non-decreasing} in the frequency of a line. Formally, suppose lines $\ell_1$ and $\ell_2$ use the same route $r$ and have frequencies $f_1$ and $f_2$, respectively. Since $\ell_1$ and $\ell_2$ share the same route $r$, we have $\Omega_{\ell_1 p} = \Omega_{\ell_2 p}$ for all $p \in \mathcal{P}$. Then, $f_1 \leq f_2 \implies v_p(\omega, \ell_1) \leq v_p(\omega, \ell_2)$ for all $\omega \in \Omega_{\ell_1 p}$.

The above formalism naturally covers trip options that {do not involve a bus segment}; in particular, we use $\omega = \varnothing$ to denote the option which consists of a passenger traveling directly from source to destination by car (the {\it no-line} option). With slight abuse of notation, we assume that $v_p(\varnothing) = 0$ for all $p \in \mathcal{P}$. Hence, one can think of the value associated with assigning a passenger to a trip option as being {relative to} the status quo ride-hailing service.

For any passenger $p$ and line $\ell$, we define the \emph{value} associated with matching the two as follows: 
\begin{definition}[Passenger-line value]
We define $\omega_{\ell p}$ and $v_{\ell p}$ to respectively be the optimal trip option, and its corresponding value, over all feasible trip options matching passenger $p$ to line $\ell$, i.e.,
\begin{align*}
v_{\ell p} = \max\left\{v_p(\omega,\ell)|\omega \in \Omega_{\ell p}\right\}, \qquad \omega_{\ell p} = \arg\max\left\{v_p(\omega,\ell)|\omega \in \Omega_{\ell p}\right\}
\end{align*}
\end{definition}
If $v_{\ell p} > 0$, we say that line $\ell$ {\it covers} passenger $p$. Let $r_{\ell p}$ denote the sub-route of $r_\ell$ used by passenger $p$ for this option. 
If $e \in r_{\ell p}$, we say that the passenger {\it uses} edge $e$. Note that computing $v_{\ell p}$ can be done in polynomial time. This follows from the fact that, if $r_{\ell}$ consists of $n_{\ell}$ edges, there are $O(n_{\ell}^2)$ possible trip options to consider for passenger $p$. Since the maximum cost (duration) of a route $D$ is constant, and $\tau_e$ is lower bounded by a constant for all $e \in E$, then $n_\ell$ is polynomial in $n$.

Using the above notation, we assume throughout that if passenger $p$ is matched to line $\ell$, she uses trip option $\omega_{\ell p}$. This assumption is primarily for the sake of simplifying the presentation; in Appendix~\ref{ssec:relaxing-trip-optimality} we discuss how our algorithm can be modified to consider all possible trip options for each line-passenger pair, and show that this only leads to an additional constant factor loss in the approximation guarantee.


\noindent\textbf{Platform Objective.}
The following example illustrates a natural value function for a platform seeking to design such an integrated mobility service.
\begin{example}\label{ex:value-function}
We abuse notation and assume that, for this example, a trip option can be parametrized by the total duration of the trip $T$ and the duration of the portion of the trip completed by car, denoted $\tcar$. Consider the following piecewise linear function, representing the reduction in time traveled by car as compared to a direct trip by car:
\begin{equation}
v_p(T,\tcar) = \begin{cases} \beta t_{s_pd_p}^\star-\tcar \quad &\text{if } T < (1+\alpha)t_{s_pd_p}^\star, \, \tcar < \beta t_{s_pd_p}^\star \\
0 &\text{otherwise}
\end{cases}
\end{equation}
where $t_{s_pd_p}^\star$ represents the time required to travel from $s_p$ to $d_p$ directly by car, $\alpha \in \mathbb{R}_+$ represents passengers' tolerance for the duration of a trip relative to the most direct route, and $\beta \in (0,1]$ controls the gains in efficiency of a trip option.

For this value function, the trip optimality assumption implies that the passenger must be picked up and dropped off at the bus stops that are closest to $s_p$ and $d_p$, respectively.
\end{example}


{Finally, in line with the motivating application of the platform receiving trip requests in advance via a scheduling service, we assume that the platform sees batch demand, and that passengers are willing to wait for the entirety of the time window. As such, we abstract away the notions of travel and clock times.} {In Appendix~\ref{ssec:travel-times} we show that such an assumption is without loss of generality, and that all results hold for a more realistic model in which there are travel times, passengers are associated with the time at which they made the request, and as a result should only be matched to lines whose schedule lines up with the time at which they are traveling.}

\subsection{The \lineplanningproblem{} (\lpp)}\label{ssec:rlpp}

Let $S \subseteq \mathcal{L}$ denote a subset of lines to be created, and $\bx \in \{0,1\}^{N\times L}$ denote an assignment of passengers to the chosen subset of lines. We first define the {\it system welfare} induced by $S$ and $\bx$.  
{\begin{definition}[Welfare]
Given $S$ and $\bx$, the {\it welfare} $W$ of the system is the sum of all passenger-line values for the lines created under this assignment. Formally:
\begin{align*}
    W = \sum_{p \in \mathcal{P}}\sum_{\ell \in S} v_{\ell p} x_{\ell p}
\end{align*}
\end{definition}}

We now define the \lineplanningproblem{}.

\begin{definition}[\lineplanningproblem{}]
The \lineplanningproblem{} is defined by a graph $G$, a set of passengers $\mathcal{P}$, costs $\{c_{\ell}\}_{\ell \in \mathcal{L}}$ for opening lines, passenger valuations $\{v_{\ell p}\}_{\ell \in \mathcal{L},p\in\mathcal{P}}$ for using each line, an overall budget $B$, and a bus capacity $C$.
The goal is to find a subset of lines to open and an assignment of passengers to lines that maximize the welfare of the system, such that:
\begin{enumerate}[$(i)$]
    \item the total cost of creating all lines in this subset does not exceed the platform's budget;
    \item the number of passengers assigned to line $\ell$ and whose trip uses edge $e \in r_{\ell}$ does not exceed the capacity $C\times f_\ell$ of the buses, for all $e \in r_{\ell}$;
    \item a passenger is assigned to at most one line (which implies no inter-bus transfers).
\end{enumerate}
\end{definition}

We allow for a passenger to not be assigned to any line. In this case, we assume that the passenger's trip is completed exclusively by car, and yields a value of zero.

Formally, the platform's optimization problem is given by:
\begin{align}
    {(P)} \qquad \max_{\by, \bx} \qquad &\sum_{p \in \mathcal{P}}\sum_{\ell \in \mathcal{L}} v_{\ell p} x_{\ell p} \notag  \\
    \text{s.t.} \qquad & \sum_{\ell \in \mathcal{L}} c_\ell y_\ell \leq B \label{eq:budget-constraint}\\
    & \sum_{\substack{p \in P:\\ e \in r_{\ell p}}} x_{\ell p} \leq C \, f_\ell \, y_\ell \quad \forall \, \ell \in \mathcal{L}, e \in r_\ell \label{eq:capacity-constraint}\\
    &\sum_{\ell \in \mathcal{L}} x_{\ell p} \leq 1 \quad \forall \, p \in \mathcal{P} \label{eq:transfer-constraint} \\
    &x_{\ell p} \in \{0,1\} \quad \forall \, p \in \mathcal{P}, \ell \in \mathcal{L} \quad \notag \\
    &y_{\ell} \in \{0,1\} \quad \forall\, \ell \in \mathcal{L} \notag
\end{align}

Let $OPT$ denote the optimal value of this optimization problem. In this formulation, the decision variables $\by \in \{0,1\}^{L}$ represent the set of lines to be opened. Recall, $\bx \in \{0,1\}^{N\times L}$ corresponds to the assignment of passengers to lines. Constraints~\eqref{eq:budget-constraint},~\eqref{eq:capacity-constraint},~\eqref{eq:transfer-constraint} respectively encode the budget, capacity, and assignment to at most one line.

For any passenger $p \in \mathcal{P}$, in the worst case there are exponentially many routes between $s_p$ and $d_p$, and as a result $(P)$ has exponentially many variables and constraints. For our main result, we make the following assumption regarding the set of routes input to \lpp.

\textbf{Assumption 2 (Candidate set of routes).} The platform has access to a pre-specified set of feasible routes that is polynomial in the size of the network. 

We let $L$ denote the size of the set of lines $\mathcal{L}$ induced by the candidate set of routes and all possible frequencies. Note that the candidate set of routes assumption implies that $L$ is polynomial in $n$. 

The assumption of such a candidate set is practically rooted in the reality of transportation systems, in which experts typically have knowledge of a priori ``acceptable'' bus routes and can develop good heuristics. Moreover, such an assumption is in line with the approach adopted in prior work on line planning, which typically generates the candidate set of routes via such heuristics~\citep{CEDER, Chakr, Fan}. In Section~\ref{sec:hardness}, we show that one cannot hope to obtain a constant-factor approximation to the \lineplanningproblem{} unless the platform has access to such a candidate set. 

{We note that the above integer linear programming (ILP) formulation problem is the most natural formulation of the platform's optimization problem, as well as the formulation upon which existing exact methods are based~\citep{wan2003mixed,barra2007solving,marin2009urban,nachtigall2008simultaneous}. In Section~\ref{sec:main-result}, we present an equivalent, less-immediate formulation of the platform's optimization problem upon which our algorithm relies. We nonetheless present this natural formulation, as we will benchmark our algorithm's performance against it in Section~\ref{sec:numerical-experiments}.}
Table~\ref{table:notation} summarizes the most frequently-used notation in the paper.

\begin{table*}[h!]
\begin{tabular}{l|l}
\textbf{Symbol} & \textbf{Definition} \\ 
\hline
$G(V,E)$ & Transit network with $|V|=n$ nodes\\
$\mathcal{L}$ & Pre-specified set of lines, with $L = |\mathcal{L}|$\\
$\mathcal{P}$ & Set of passengers, with $N = |\mathcal{P}|$ \\
$\Omega_{\ell p}$ & Set of feasible trip options for passenger $p$ traveling via line $\ell$ \\
$C$ & Bus capacity \\
$B$ & Platform budget for opening lines \\
$v_{p}(\omega, \ell)$ & Value of trip option $\omega \in \Omega_{\ell p}$ for passenger $p$ traveling via line $\ell$\\ 
$v_{\ell p}$ & Value of optimal trip option for passenger $p$ on line $\ell$\\ 
$c_{\ell}$ & Cost of opening line $\ell$ \\
$f_{\ell}$ & Frequency of line $\ell$ \\
\hline
\end{tabular}
\caption{List of frequently-used notations}
\label{table:notation}
\end{table*}

\section{Fundamental limits of real-time routing}
\label{sec:hardness}

The model in Section~\ref{sec:preliminaries} is endowed with two assumptions: $(i)$ the existence of a pre-specified \emph{candidate set} of feasible lines $\mathcal{L}$ that is polynomial in the number of nodes $n$, and $(ii)$ that trip options can involve at most a single bus segment. 

In this section, we show that these assumptions are not just practically relevant, but also have strong theoretical justifications: if either assumption fails to hold, a constant-factor approximation is out of reach. We moreover show that, even in the setting where these two assumptions hold, standard approximation techniques that leverage naive LP relaxations and rely on submodularity are inadequate, emphasizing the non-triviality of the task of designing provably good approximations for fast, real-time routing. 

In the remainder of this section, we provide the main ideas of our reductions, and defer proofs of all auxiliary propositions to Appendix~\ref{app:hardness-proofs}.

\subsection{Necessity of a candidate set of lines}

Suppose first that the \platform{} does not have access to a candidate set of lines, and thus, for each passenger $p \in \mathcal{P}$, must consider all possible walks of bounded cost between source $s_p$ and destination $d_p$. We show that this problem is hard to approximate even in a particularly simple instance of \lpp\ with only a single allowed line, which we term the Single Line Problem (\textsc{Slp}).

\begin{definition}[Single Line Problem]
In the Single Line Problem, the feasible routes are the walks in the graph of cost at most $D$. Suppose $c_{\ell} = cf_{\ell}$ for all $\ell \in \mathcal{L}$, for some constant $c > 0$. Moreover, suppose $B = c$. That is, only a single line of frequency $f_\ell = 1$ can be opened. The goal is to find the line that maximizes the social welfare of the system.
\end{definition}

Using this, we get our first hardness result for \lpp.

\begin{theorem}\label{thm:necessity-of-candidate-set}
Unless NP has polynomial Las Vegas algorithms, the Single Line Problem is hard to approximate to a ratio better than $\Omega(\log^{1-\varepsilon} n)$.
\end{theorem}

To establish this inapproximability result, we give a reduction from the Orienteering group TSP problem (\textsc{OgTSP}), for which the approximation lower bound is $\Omega\left(\log^{1-\varepsilon} n\right)$~\citep{recursive_greedy}. 

\begin{definition}[Orienteering group TSP]
Given an undirected graph $G=(V,E)$, with edge costs $w: E \mapsto \mathbb{R}_+$, $k$ sets (or groups) of vertices $S_{1},\ldots, S_{k} \subseteq V$, a root vertex $r$ and a budget $D>0$, the goal is to find a walk of cost no more than $D$ which spans the maximum number of groups.\footnote{We assume without loss of generality that the root does not belong to any of the groups.}
\end{definition}

\begin{proof}[Proof of Theorem~\ref{thm:necessity-of-candidate-set}]
Consider an instance of \textsc{OgTSP}. Recall, we've assumed that there exists a constant $\tau_{\min} > 0$ such that $\tau_e > \tau_{\min} \, \forall \, e \in E$. Define $\varepsilon \in (0, \tau_{\min}]$. We use $diam(G)$ to denote the diameter of the graph, and let $t \in \mathbb{R}$ be such that $t > \max\{ diam(G)+\varepsilon, D + \varepsilon\}$.

We construct an instance of \textsc{Slp} as follows. For each group $S_i$, we add a node $g_i$ to $G$, an edge $(r,g_i)$ of cost $t$ and an edge $(j,g_i)$ of cost $t-\varepsilon$ for each node $j \in S_i$. Let $G' = (V',E')$ denote this augmented graph, and let $D$ be the maximum cost of any feasible route on $G'$. For each $i \in [k]$, create a passenger $p_i$ with $s_{p_i} = r$ and $d_{p_i} = g_i$. 

For line-passenger pair $(\ell,p_i)$, suppose trip option $\omega$ is such that passenger $p_i$ travels by car from $r$ to $j_1(\omega)$, and from $j_2(\omega)$ to $g_i$, where $j_1(\omega), j_2(\omega) \in V.$ We use $t^{\text{car}}(\omega)$ to denote the total cost of the min-cost paths from $r$ to $j_1(\omega)$ and from $j_2(\omega)$ to $g_i$, and let $t^\star_{p_i}$ denote the min-cost path from $r$ to $g_i$. If $p_i$ travels directly from $r$ to $g_i$ via edge $(r,g_i)$, then $t^{\text{car}}(\omega) = t$.

We define the value function as follows:
\begin{align*}
    v_{p_i}(\omega,\ell) = \begin{cases}
    1 &\quad \text{ if } t^{\text{car}}(\omega) \leq (1-\frac{\varepsilon}{t})t_{p_i}^\star \\
    0 &\quad \text{ otherwise.}
    \end{cases}
\end{align*}

\begin{figure}\label{fig:two-lines}
    \centering
    \includegraphics[scale=0.12]{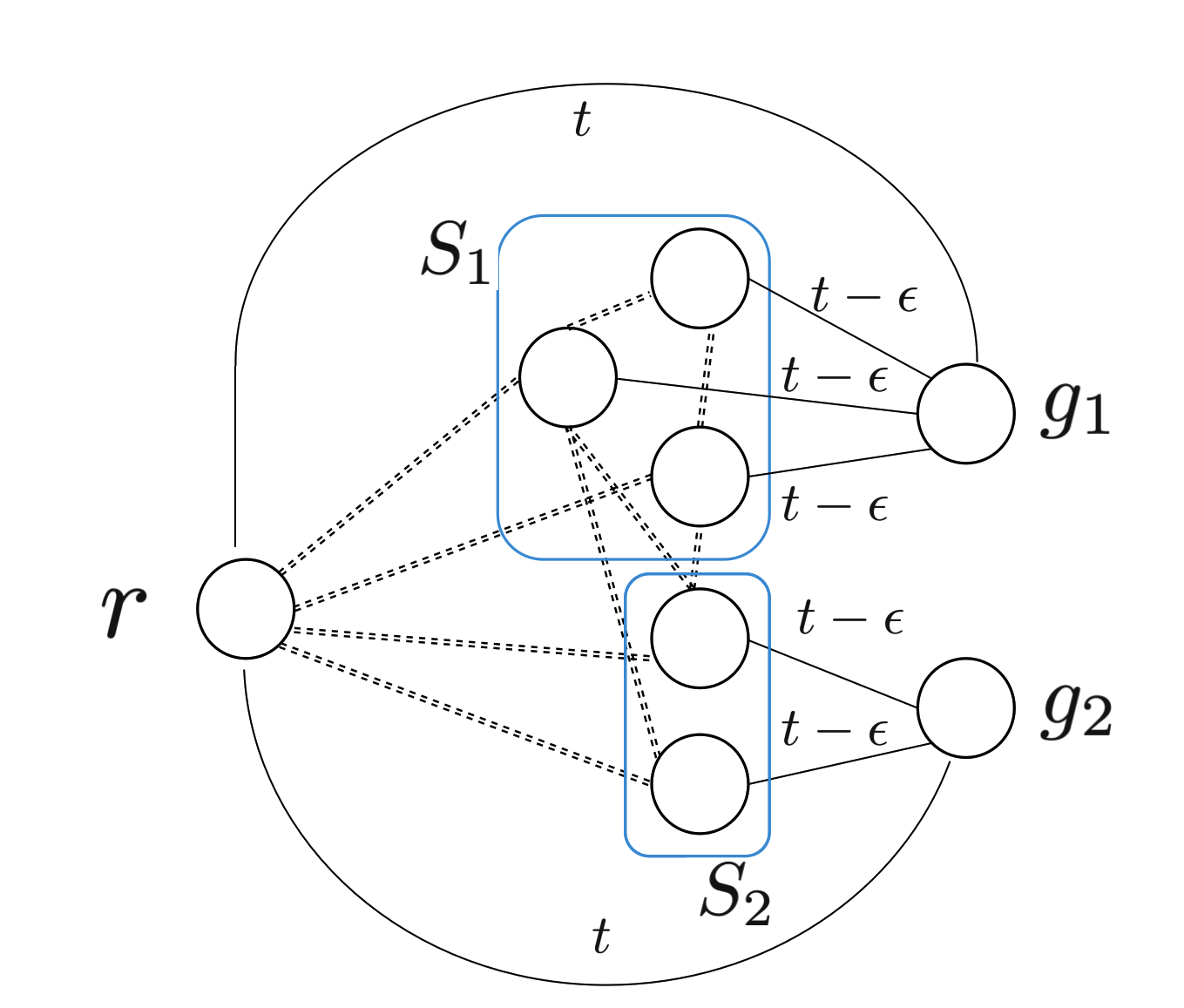}
    \caption{Construction of graph $G'$ from an instance of \textsc{OgTSP} with two groups $S_1$ and $S_2$. The dashed lines represent the edges of the original graph $G$.}
\end{figure}

 Propositions~\ref{prop:helper} and~\ref{prop:hardness-helper} characterize the ways in which $p_i$ can feasibly travel from $r$ to $g_i$.
{\begin{proposition}\label{prop:helper}
For all $\omega$ such that $t^{\text{car}}(\omega) > t-\varepsilon$, $v_{p_i}(\omega,\ell) = 0$. 
\end{proposition}}

\begin{proposition}\label{prop:hardness-helper}
Passenger $p_i$ can travel from $r$ to $g_i$ in one of two ways:
\begin{enumerate}[(i)]
    \item via edge $(r, g_i)$, in which case this must be by car.
    \item by bus from $r$ to $j \in S_i$, and by car via edge $(j,g_i)$.
\end{enumerate}
\end{proposition}

Let $\ell^\star$ denote the optimal solution to \textsc{Slp} for this instance.
\begin{proposition}\label{prop:collect-pos-val}
To collect strictly positive value from passenger $p_i$, $\ell^\star$ {\it must} traverse a node $j \in S_i$.
\end{proposition}

Finally, observe that $\ell^\star$ necessarily only uses edges from $E$. This follows from the fact that all edges in $E'\setminus E$ have cost greater than $D$ by construction, and thus any route using at least one such edge is infeasible. 

Putting these facts together, if line $\ell^\star$ collects value $k' \leq k$ then this implies the existence of a walk of $G$ of cost at most $D$ that has visited $k'$ groups. Thus any $\alpha$-approximation algorithm for the Single Line Problem gives an $\alpha$-approximation for the \textsc{OgTSP}, hence the $\Omega(\log^{1-\varepsilon}(n))$ lower bound for the Single Line Problem.
\end{proof}

\subsection{Hardness of multiple transfers}

Suppose now that the \platform{} has access to a candidate set of lines, but allows itself to assign passengers to {at most} {\it two} lines. More specifically, a passenger $p$ can feasibly be assigned to the following trip options:
\begin{enumerate}[$(i)$]
    \item Travel directly from $s_p$ to $d_p$ by car;
    \item Use a single bus line $\ell \in \mathcal{L}$: for some $v_1 \in r_{\ell}, v_2 \in r_{\ell}$, travel from $s_p$ to $v_1$ by car; join line $\ell$ at $v_1$ and travel to $v_2$ by bus; travel from $v_2$ to $d_p$ by car;
    \item Use two intersecting bus lines $(\ell_1,\ell_2) \in \mathcal{L}\times\mathcal{L}$: for some $v_1 \in r_{\ell_1}, v_2 \in r_{\ell_1}\bigcap r_{\ell_2}, v_3 \in r_{\ell_2}$, travel from $s_p$ to $v_1$ by car; join line $\ell_1$ at $v_1$ and travel to $v_2$ by bus; join line $\ell_2$ at $v_2$ and travel to $v_3$ by bus; travel from $v_3$ to $d_p$ by car. Figure~\ref{fig:two-lines} illustrates such a trip. {We use $\Omega_{(\ell_1,\ell_2),p}$ to denote the set of all such trips.}
\end{enumerate}

Let $v_{(\ell_1,\ell_2),p}$ denote the maximum value passenger $p$ has for all feasible trips using lines $\ell_1$ and $\ell_2$, where $r_{\ell_1}$ and $r_{\ell_2}$ intersect. That is, $v_{(\ell_1,\ell_2),p}= \max\limits_{\substack{\omega \in \Omega_{(\ell_1,\ell_2),p}}} v_{p}(\omega)$. If $v_{(\ell_1,\ell_2),p} > 0$, we say that passenger $p$ is {\it covered} by $\ell_1$ and $\ell_2$.\

\begin{figure}[!t]
\centering
\tikzset{every picture/.style={line width=0.75pt}} 

\begin{tikzpicture}[x=0.75pt,y=0.75pt,yscale=-1,xscale=1]

\draw    (120,90) -- (160,130) ;
\draw [shift={(160,130)}, rotate = 45] [color={rgb, 255:red, 0; green, 0; blue, 0 }  ][fill={rgb, 255:red, 0; green, 0; blue, 0 }  ][line width=0.75]      (0, 0) circle [x radius= 3.35, y radius= 3.35]   ;
\draw [shift={(120,90)}, rotate = 45] [color={rgb, 255:red, 0; green, 0; blue, 0 }  ][fill={rgb, 255:red, 0; green, 0; blue, 0 }  ][line width=0.75]      (0, 0) circle [x radius= 3.35, y radius= 3.35]   ;
\draw    (160,130) -- (200,170) ;
\draw    (160,130) -- (230,100) ;
\draw [shift={(230,100)}, rotate = 336.8] [color={rgb, 255:red, 0; green, 0; blue, 0 }  ][fill={rgb, 255:red, 0; green, 0; blue, 0 }  ][line width=0.75]      (0, 0) circle [x radius= 3.35, y radius= 3.35]   ;
\draw    (260,90) -- (230,100) ;
\draw    (90,160) -- (160,130) ;
\draw [shift={(160,130)}, rotate = 336.8] [color={rgb, 255:red, 0; green, 0; blue, 0 }  ][fill={rgb, 255:red, 0; green, 0; blue, 0 }  ][line width=0.75]      (0, 0) circle [x radius= 3.35, y radius= 3.35]   ;
\draw    (100,70) -- (120,90) ;
\draw [color={rgb, 255:red, 3; green, 3; blue, 226 }  ,draw opacity=1 ]   (139,63) -- (121.12,80.88) ;
\draw [shift={(119,83)}, rotate = 315] [fill={rgb, 255:red, 3; green, 3; blue, 226 }  ,fill opacity=1 ][line width=0.08]  [draw opacity=0] (5.36,-2.57) -- (0,0) -- (5.36,2.57) -- cycle    ;
\draw [color={rgb, 255:red, 3; green, 3; blue, 226 }  ,draw opacity=1 ]   (234,93) -- (221.17,62.76) ;
\draw [shift={(220,60)}, rotate = 427.01] [fill={rgb, 255:red, 3; green, 3; blue, 226 }  ,fill opacity=1 ][line width=0.08]  [draw opacity=0] (5.36,-2.57) -- (0,0) -- (5.36,2.57) -- cycle    ;
\draw  [dash pattern={on 4.5pt off 4.5pt}]  (130,92) -- (157.81,117.95) ;
\draw [shift={(160,120)}, rotate = 223.03] [fill={rgb, 255:red, 0; green, 0; blue, 0 }  ][line width=0.08]  [draw opacity=0] (5.36,-2.57) -- (0,0) -- (5.36,2.57) -- cycle    ;
\draw  [dash pattern={on 4.5pt off 4.5pt}]  (168,121) -- (215.21,102.11) ;
\draw [shift={(218,101)}, rotate = 518.2] [fill={rgb, 255:red, 0; green, 0; blue, 0 }  ][line width=0.08]  [draw opacity=0] (5.36,-2.57) -- (0,0) -- (5.36,2.57) -- cycle    ;

\draw (101,80.4) node [anchor=north west][inner sep=0.75pt]  [font=\small]  {$v_{1}$};
\draw (151,137.4) node [anchor=north west][inner sep=0.75pt]  [font=\small]  {$v_{2}$};
\draw (221,103.4) node [anchor=north west][inner sep=0.75pt]  [font=\small]  {$v_{3}$};
\draw (141,47.4) node [anchor=north west][inner sep=0.75pt]  [font=\small]  {$s$};
\draw (223,47.4) node [anchor=north west][inner sep=0.75pt]  [font=\small]  {$d$};
\draw (204,159.4) node [anchor=north west][inner sep=0.75pt]  [font=\small]  {$\ell _{1} \ $};
\draw (71,143.4) node [anchor=north west][inner sep=0.75pt]  [font=\small]  {$\ell _{2} \ $};

\end{tikzpicture}
\caption{Assignment of a passenger to a pair of lines. The passenger travels by car from $s$ to $v_1$. Between $v_1$ and $v_2$, she travels by bus via line $\ell_1$. At $v_2$ she travels via line $\ell_2$ until being dropped off at $v_3$. She completes her trip by car between $v_3$ and $d$.}
\label{fig:two-lines}
\end{figure}
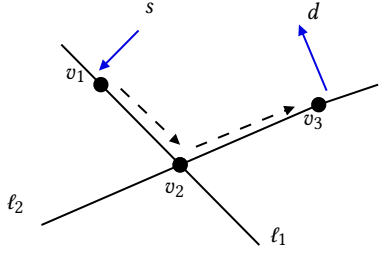

We refer to the problem of matching passengers to at most two bus lines as the {\it Two-Transfer Problem} (\textsc{Ttp}), which we formally define below.
\begin{definition}[Two-Transfer Problem]
Given a budget $B$ and costs $\{c_\ell\}$, the goal is to find a subset $S \subseteq \mathcal{L}$ of budget-respecting lines to open and a feasible assignment of passengers to $S$ which maximizes the social welfare of the system, given by:
\begin{equation*}
    \sum_{p\in \mathcal{P}}\left( \sum_{\ell\in S} v_{\ell p} x_{\ell p} + \sum_{(\ell_1,\ell_2)\in S\times S} v_{(\ell_1,\ell_2),p} \, x_{(\ell_1,\ell_2),p}\right).
\end{equation*}
As before, $\bx$ is an indicator variable representing the assignment of passengers to lines.
\end{definition}


Our next hardness result shows that allowing even two inter-bus transfers banishes any hope of obtaining a constant-factor approximation for \lpp.

\begin{theorem}\label{thm:hardness-many-transfers}
Under the exponential time hypothesis, the Two-Transfer Problem is hard to approximate to a ratio better than $\Omega\left(n^{{1}/{(\log\log(n))^c}}\right)$, where $c > 0$ is a universal constant.
\end{theorem}

To prove the theorem, we give a reduction from the densest $k$-subgraph problem, {which admits an approximation lower bound of $\Omega(n^{1/(\log\log n)^c})$ under the exponential time hypothesis~\citep{manurangsi2017almost}.} Given a graph $G=(V,E)$ and a subgraph $G_s=(V_s,E_s)$ of $G$, the density of any subgraph $G_s$ is the ratio of number of edges to the number of nodes in $G_s$ (i.e.$\frac{|E_s|}{|V_s|}$). Now, the densest $k$-subgraph problem is as follows:
\begin{definition}[Densest $k$-subgraph]
Given a graph $G=(V,E)$ with $n = |V|$ and $k\in [n]$, the objective is to find a subgraph $G_s$ of $G$ containing {exactly} $k$ vertices with maximum density.
\end{definition}

{Note that, for fixed $k$, finding the subgraph of maximum density is equivalent to finding a subgraph of size $k$ with the maximum number of edges.}

\begin{proof}[Proof of Theorem~\ref{thm:hardness-many-transfers}]
Given an instance of densest $k$-subgraph, we build an instance of \textsc{Ttp} as follows. For each node $i \in V$, construct a line $\ell_i$, with $c_{\ell_i} = 1$ and frequency $f_{\ell_i}$ large enough to cover all passengers. For every edge $(i,j) \in E$, define a passenger $p_{ij}$, and suppose that $p_{ij}$ can only be covered by the pair of lines $(\ell_i, \ell_j)$, with $v_{(\ell_i,\ell_j),p_{ij}} = 1$. That is, $p_{ij}$ {\it has no value associated with a single bus line.} Finally, let $B = k$. 

We first claim that, for any \textsc{Ttp} feasible solution of value $k'$ which opens $k'' < k$ lines, one can construct a feasible solution which opens {\it exactly} $k$ lines and has value at least $k'$. This simply follows from non-negativity of the value function and the fact that $c_{\ell_i} = 1$ for all $i$. Thus, the platform can always open $k-k''$ more lines until hitting its budget constraint and not decrease the objective, and it is without loss of generality to only consider feasible solutions that open exactly $k$ lines.

We complete the proof by noting that a feasible solution of value $k'$ corresponds exactly to a subgraph of $G$ containing $k'$ edges (passengers) and $k$ nodes (lines). Thus, if we had a constant-factor approximation algorithm for \textsc{Ttp}, then we would also be able to approximate densest $k$-subgraph within a constant factor.
\end{proof}

Henceforth, we operate under the no inter-bus transfers and candidate set of lines assumptions.

\subsection{Inefficacy of standard approximation techniques}

Observe that the ILP formulation of the \lineplanningproblem{} bears a strong resemblance to the Capacitated Facility Location Problem (\textsc{Cflp}), for which~\citet{Wolsey} provides a $1-\frac1e$ approximation algorithm, a guarantee relying on the underlying {\it submodular} structure of \textsc{Cflp}. Our problem crucially differs from this latter problem, however, in the way capacity is accounted for. Whereas the number of clients assigned to a location cannot exceed its capacity in \textsc{Cflp}, in the \lineplanningproblem{} the number of passengers assigned to a bus {\it can exceed} its capacity, as passengers may require non-overlapping subpaths of a bus route. In this section, we show that this simple fact fundamentally alters the structure of our problem, and as such precludes the use of standard techniques for submodular function maximization. 

Let $w: \{0,1\}^{L} \mapsto \mathbb{R}$ denote the social welfare induced by the optimal assignment of passengers to lines, for a given subset of open lines, represented by $\by$. Formally:
\begin{align*}
    w(\by) = \max_{\bx} \qquad &\sum_{p \in \mathcal{P}}\sum_{\ell \in \mathcal{L}} v_{\ell p} x_{\ell p} \notag \\
    \text{s.t.} \qquad
    & \sum_{\substack{p \in P:\\ e \in r_{\ell p}}} x_{\ell p} \leq C \, f_\ell \, y_\ell \quad \forall \, \ell \in \mathcal{L}, e \in r_{\ell} \notag \\
    &\sum_{\ell \in \mathcal{L}} x_{\ell p} \leq 1 \quad \forall \, p \in \mathcal{P} \notag \\
    &x_{\ell p} \in \{0,1\} \quad \forall \, p \in \mathcal{P}, \ell \in \mathcal{L} \notag
\end{align*}

Then, we have:
\begin{align*}
    OPT = \max_{\by} \qquad &w(\by) \\
    \text{s.t.} \qquad & \sum_{\ell \in \mathcal{L}} c_\ell y_\ell \leq B\\
    &y_{\ell} \in \{0,1\} \quad \forall\, \ell \in \mathcal{L} \notag
\end{align*}

\begin{proposition}\label{prop:not-submodular}
$w$ is not submodular.
\end{proposition}

Another common approach is to develop an approximation algorithm based on an LP relaxation of the ILP. Proposition~\ref{prop:integrality-gap} however shows that such an approach can give strictly worse bounds than the $1-\frac1e$ benchmark.

\begin{proposition}\label{prop:integrality-gap}
The worst-case integrality gap for $(P)$ is no better than $\frac12$.
\end{proposition}

\section{Main result}
\label{sec:main-result}

In this section, we design an approximation algorithm for the \lineplanningproblem{} that achieves at least $1-\frac1e-\varepsilon$ fraction of the optimal solution in expectation, and produces a solution whose cost is budget-respecting with high probability, as the \platform's budget grows large.

{Our high-level approach is as follows. We first formulate the \lineplanningproblem{} as a configuration ILP, and solve a conservative LP relaxation of this latter program, in the sense that it has a stricter budget than the \platform's true budget $B$. We then use a variant of the rounding scheme developed by~\citet{SAP} to produce an approximately feasible integer solution. The key difficulty in such an approach is approximating the exponential-size configuration LP without incurring too much of a loss. Our main contribution in this respect is to show that the structure of \lpp{} allows us to solve it {\it exactly} in polynomial-time by leveraging the additional structure of our problem in the dual space. Throughout the rest of the section, we defer the proofs of auxiliary facts to Appendix~\ref{app:main-result-proofs}.}


\subsection{An exponential-size configuration ILP}


Consider line $\ell$, and let $\mathcal{I}_{\ell}$ denote the family of all feasible assignments of passengers to $\ell$, where a feasible assignment is such that, for all $e \in r_{\ell}$ the total number of passengers using $e$ does not exceed the capacity of the line. We use $S$ to denote any such assignment in $\mathcal{I}_{\ell}$. $X_{\ell S}$ is the indicator variable representing whether or not the set of passengers $S$ is chosen for line $\ell$. Formally, $S \in \mathcal{I}_{\ell}$ satisfies $\sum\limits_{\substack{p \in S:\\e \in r_{\ell p}}} X_{\ell S} \leq C f_{\ell}$ for all $e \in E$. {Example~\ref{ex:SAP} illustrates this notation.

\begin{example}\label{ex:SAP}
Consider lines $\ell_1, \ell_2$ and passengers $p_1, p_2$, with $p_1$ and $p_2$ using the same edges of each line. If $C = 2$, then $\mathcal{I}_{\ell_i} = \left\{\{p_1\}, \{p_2\}, \{p_1, p_2\}\right\}$ for $i \in \{1,2\}$. If $C = 1$,  then $\mathcal{I}_{\ell_i} = \left\{\{p_1\}, \{p_2\}\right\}$ for $i \in \{1,2\}$.
\end{example}

}


We can now represent {\lpp} as the following exponential-size integer program:
\begin{align}
{\widehat{P}:=} \qquad \max_{\left\{X_{\ell S}\right\}} \qquad
&  \sum_{p \in \mathcal{P}} \; \sum_{\ell\in \mathcal{L}} \sum_{\substack{S \in \mathcal{I}_{\ell}:\\p\in S}}v_{\ell p} X_{\ell S}  \notag\\
 \text{s.t.}\qquad & \sum_{\ell\in \mathcal{L}} c_\ell\left( \sum_{S\in \mathcal{I}_\ell} X_{\ell S}\right) \leq B \label{budget-config-lp-1}\\
 &    \sum_{S\in \mathcal{I}_\ell} X_{\ell S} \leq 1\qquad\qquad \forall \, \ell \in \mathcal{L} \label{one-set-per-line-1} \\
 &\sum_{\ell\in \mathcal{L}} \sum_{\substack{S \in \mathcal{I}_{\ell}:\\p\in S}} X_{\ell S} \leq 1 \qquad \forall \, p \in \mathcal{P} \label{cust-to-one-set-1} \\
&  X_{\ell S} \in \{0,1\} \qquad\qquad \forall \, \ell\in \mathcal{L}, S \in \mathcal{I}_{\ell}\notag
\end{align}

Constraint~\eqref{one-set-per-line-1} requires that only one set of passengers be chosen for each line, and Constraint~\eqref{cust-to-one-set-1} ensures that each passenger is only assigned to one line. If a set of passengers is assigned to line $\ell$, that is, if $\sum_{S \in \mathcal{I}_{\ell}} X_{\ell S} > 0$, then $\ell$ is opened and the \platform{} incurs cost $c_{\ell}$; else, $\ell$ is not created and no cost is incurred. Let $OPT$ denote the optimal value of $\widehat{P}$.


\subsection{Approximating the exponential-size ILP}

{For a given constant $\varepsilon \in (0,\frac12)$, Algorithm~\ref{alg:rounding} makes use of the following auxiliary configuration LP, which we denote $\widehat{P}^{(\varepsilon)}$.}

{\begin{align}
{\,\widehat{P}^{(\varepsilon)}\,:=} \qquad \max_{\left\{X_{\ell S}\right\}} \qquad
&  \sum_{p \in \mathcal{P}} \; \sum_{\ell\in \mathcal{L}} \sum_{\substack{S \in \mathcal{I}_{\ell}:\\p\in S}}v_{\ell p} X_{\ell S}  \notag\\
 \text{s.t.}\qquad & \sum_{\ell\in \mathcal{L}} c_\ell\left( \sum_{S\in \mathcal{I}_\ell} X_{\ell S}\right) \leq B(1-\varepsilon) \label{budget-constraint-relaxation}\\
 &    \sum_{S\in \mathcal{I}_\ell} X_{\ell S} \leq 1\qquad\qquad \forall \, \ell \in \mathcal{L} \label{one-set-per-line-relaxation} \\
 &\sum_{\ell\in \mathcal{L}} \sum_{\substack{S \in \mathcal{I}_{\ell}:\\p\in S}} X_{\ell S} \leq 1 \qquad \forall \, p \in \mathcal{P} \label{one-passenger-per-set-relaxation}  \\
&  X_{\ell S} \in [0,1] \qquad\qquad \forall \, \ell\in \mathcal{L}, S \in \mathcal{I}_{\ell}\notag
\end{align}}

Let ${OPT}^{(\varepsilon)}$ denote the optimal value of $\widehat{P}^{(\varepsilon)}$, and $\left\{X_{\ell S}^{(\varepsilon)}\right\}$ its optimal solution.  
Algorithm~\ref{alg:rounding} presents a high-level description of our algorithm. 

\begin{algorithm}
\begin{algorithmic}
\Require {$G = (V,E), \mathcal{P},\mathcal{L},\left\{\mathcal{I}_{\ell}\right\}_{\ell\in\mathcal{L}},  \varepsilon \in (0,\frac12)$}
\Ensure {set of lines to open, passenger assignment to each line}
\State Compute $v_{\ell p}$ for all $\ell \in \mathcal{L}, p \in \mathcal{P}$.
\State Solve $\widehat{P}^{(\varepsilon)}$.
\State \textbf{Rounding:} For all $\ell \in \mathcal{L}, S \in \mathcal{I}_{\ell}$ such that $X_{\ell S}^{(\varepsilon)} > 0$, {open $\ell$ and}, independently for each line $\ell$, assign $S$ to $\ell$ with probability $X_{\ell S}^{(\varepsilon)}$.
\State \textbf{Re-assignment:} If passenger $p$ is  assigned to multiple lines, choose the line maximizing $v_{\ell p}$. Close all lines for which no passengers are any longer assigned.
\State \textbf{Aggregation:} {If there exist open lines $\ell_1, \ell_2$ such that $r_{\ell_1} = r_{\ell_2} = r$ and $f_{\ell_1} \neq f_{\ell_2}$, close $\ell_1$ and $\ell_2$ and open $\ell^\prime = (r, f_{\ell_1} + f_{\ell_2}).$ Assign all passengers formerly using $\ell_1$ or $\ell_2$ to $\ell^{\prime}$.}
\end{algorithmic}
	\caption{Randomized rounding for \lpp\label{alg:rounding}}
\end{algorithm}

Let $ALG$ denote the expected value of the solution returned by Algorithm~\ref{alg:rounding}. Theorem~\ref{thm:main-thm} establishes our main result.

\begin{theorem}\label{thm:main-thm}
Algorithm~\ref{alg:rounding} respects the budget in expectation, and is of cost no more than $B$ with probability at least $1 -e^{-\varepsilon^2B/3c_{\max}}$,
where  $c_{\max} = \max_{\ell \in \mathcal{L}} c_\ell$. Moreover,
$$ALG \geq \left(1-\frac1e-\varepsilon\right)OPT.$$
\end{theorem}

{Note that the choice of $\varepsilon$ trades off between quality of approximation and feasibility of the rounded solution: as $\varepsilon$ increases, the solution is exponentially more likely to be budget-respecting; on the other hand, we lose $\varepsilon$-fraction of the optimum in terms of the approximation guarantee.}

{To prove Theorem~\ref{thm:main-thm}, we establish the following facts, which characterize the loss incurred in each step of the algorithm:
\begin{enumerate}[$(i)$]
    \item\label{f5} $OPT^{(\varepsilon)} \geq (1-\varepsilon)OPT$ (Proposition~\ref{prop:loss-from-epsilon}).
    \item $\widehat{P}^{(\varepsilon)}$ can be solved in polynomial time (Theorem~\ref{thm:hardness-many-transfers});
    \item\label{f4} the loss from rounding and re-assignment is at most $\frac1e$ fraction of the optimal value of $\widehat{P}^{(\varepsilon)}$ (Proposition~\ref{thm:rounding_value});
    \item\label{f1} the aggregation step maintains a feasible assignment of passengers to lines, and neither increases the cost of the solution nor decreases the objective (Proposition~\ref{prop:final-step-doesnt-matter});
    \item\label{f2} the cost of the final solution respects the \platform's budget with high probability (Corollary~\ref{cor:cost});
\end{enumerate}}

We first show that the loss incurred from solving the auxiliary LP is not too large.
\begin{proposition}\label{prop:loss-from-epsilon}
For all $\varepsilon \in [0,1]$, $$OPT^{(\varepsilon)} \geq (1-\varepsilon)OPT.$$
\end{proposition}

\begin{proof}
Let $\{X_{\ell S}^{(0)}\}$ denote the optimal solution to $\widehat{P}^{(0)}$. Observe that $\{(1-\varepsilon)X_{\ell S}^{(0)}\}$ is feasible for the problem $\widehat{P}^{(\varepsilon)}$, and that the objective of $\widehat{P}^{(\varepsilon)}$ evaluated at this feasible solution is: $$(1-\varepsilon)\sum_{p \in \mathcal{P}}\sum_{\ell\in\mathcal{L}}\sum_{\substack{S \in \mathcal{I}_{\ell}:\\ p \in S}} v_{\ell p}X_{\ell S}^{(0)} = (1-\varepsilon) OPT^{(0)}$$

Observe moreover that $\widehat{P}^{(0)}$ corresponds to the LP relaxation of $\widehat{P}$, and thus $OPT^{(0)} \geq OPT$. Chaining these two inequalities together we obtain the fact. 
\end{proof}

We next observe that Algorithm~\ref{alg:rounding} is underdetermined as defined. In particular, it is a priori unclear how, if at all, one can efficiently solve $\widehat{P}^{(\varepsilon)}$ in polynomial time, or if the best we can hope for is an approximation. Our key contribution is showing that this can in fact efficiently be done, and as a result the only losses potentially incurred by the algorithm come from the rounding, re-assignment, and aggregation steps.
\begin{theorem}\label{thm:poly-time-separation}
$\widehat{P}^{(\varepsilon)}$ can be solved in polynomial time.
\end{theorem}


\begin{proof}
Since $\widehat{P}^{(\varepsilon)}$ has an exponential number of variables but only a polynomial number of constraints (in the number of passengers and lines, and hence in $n$), its dual has polynomially many variables, and as such can be solved in polynomial time via the ellipsoid method, {\it assuming access to a polynomial-time separation oracle}~{\citep{bland1981ellipsoid}}. {Given this, one can obtain an optimal primal solution by solving the primal problem with only the variables corresponding to the dual constraints present when the ellipsoid method has terminated (of which there are polynomially many, since the ellipsoid method only makes a polynomial number of calls to the separation oracle)~\citep{carr2000randomized}. Thus, it suffices to design a separation oracle which runs in polynomial time.} 

Let $\widehat{D}^{(\varepsilon)}$ denote the dual of $\widehat{P}^{(\varepsilon)}$, with $\alpha, \{q_{\ell}\}, \{\lambda_p\}$ the dual variables corresponding{ to constraints~\eqref{budget-constraint-relaxation}, \eqref{one-set-per-line-relaxation} and \eqref{one-passenger-per-set-relaxation}, respectively.} The dual is given by:
\begin{align*}
\widehat{D}^{(\varepsilon)} := \qquad \min_{\substack{\{q_{\ell}\}, \{\lambda_p\}, \alpha}}
 \qquad & \sum_{\ell \in \mathcal{L}} q_\ell + \sum_{p \in \mathcal{P}} \lambda_p +B(1-\varepsilon)\alpha\\
 \text{s.t.} \qquad  &   q_\ell + \alpha c_\ell \geq \sum_{p\in S} \left(v_{\ell p} - \lambda_p\right) \qquad \forall \, \ell \in \mathcal{L}, S\in \mathcal{I}_\ell \\
 &    q_\ell \geq 0 \quad \forall \, \ell \in \mathcal{L},\quad \lambda_p \geq 0 \quad \forall \, p \in \mathcal{P}, \quad \alpha \geq 0
\end{align*}

For all $\ell\in\mathcal{L}$, let $\mathcal{F}_{\ell}$ denote the polytope defined by the set of constraints:
\begin{equation*}
    q_\ell + \alpha c_\ell \geq \sum_{p\in S} (v_{\ell p}-\lambda_p) \qquad \forall \, S\in \mathcal{I}_\ell
\end{equation*}
It suffices to show that we can design a polynomial time separation algorithm for the polytope $\mathcal{F}_{\ell}$. That is, given $q_{\ell}, \alpha,$ and $\left\{\lambda_p\right\}$, the separation algorithm must be able to find a violated constraint for $\mathcal{F}_\ell$ or certify that all constraints in $\mathcal{F}_\ell$ are satisfied.

Algorithm~\ref{alg:separation} formally describes our separation oracle.

\begin{algorithm}
\begin{algorithmic}
\Require {$q_{\ell}, \alpha, \{\lambda_p\}, \mathcal{F}_{\ell}$}
\Ensure {violated constraint for $\mathcal{F}_{\ell}$, or a certification that all constraints in $\mathcal{F}_{\ell}$ are satisfied}
\State Solve the following LP:
\begin{align}
   \max_{\{x_p\}} \qquad & \sum_{p \in \mathcal{P}} (v_{\ell p}-\lambda_p) x_p\notag\\\
\text{s.t.} \qquad &  \sum_{\substack{p \in \mathcal{P}: \\ e \in r_{\ell p}}} x_p \leq Cf_{\ell} \qquad \, \forall \, e \in r_{\ell} \label{single_line_subproblem}\\ &   0 \leq x_p \leq 1 \qquad \forall \, p \in \mathcal{P}.\notag
\end{align}
Let LP-SEP denote its optimal value, and $\{x_p^\star\}$ an optimal solution to this problem.
\State If LP-SEP $\leq q_{\ell} + \alpha c_{\ell}$, then return that all constraints in $\mathcal{F}_{\ell}$ are satisfied. Else, return $S^\star = \{p:x_{p}^\star > 0\}$.
\end{algorithmic}
	\caption{Separation Algorithm for the Ellipsoid Method\label{alg:separation}}
\end{algorithm}

Our separation algorithm solves an LP with polynomially many variables and constraints, and as such runs in polynomial time.{\footnote{We note that, given a dual solution, one can efficiently find a primal solution, as observed by~\citet{carr2000randomized}.}} However, correctness of the algorithm is not immediate: the LP is a relaxation of the set problem we are interested in, and as such $\sum_{p} \left(v_{\ell p} - \lambda_p\right) x_p^\star \geq \max_{S \in \mathcal{I}_{\ell}} \sum_{p \in S}\left(v_{\ell p} - \lambda_p\right)$. If this inequality was strict, the separation algorithm would incorrectly return that a constraint has been violated, when in fact all have been satisfied. {Observe that this would only occur if $\{x^\star_p\}$ were fractional; the separation algorithm we propose, however, is a capacitated variant of the assignment problem, for which the linear programming relaxation is known to admit an integral solution~\citep{bertsekas1991linear}.} Lemma~\ref{lem:correctness} formalizes this high-level intuition, and thus establishes that this inequality is in fact always tight. 

This then concludes the proof of the fact that $\widehat{P}^{(\varepsilon)}$ is poly-time solvable.
\begin{lemma}\label{lem:correctness}
$\{x_p^\star\}$ is integral. Thus,
$$\sum_{p} \left(v_{\ell p} - \lambda_p\right)x_p^\star = \max_{S \in \mathcal{I}_{\ell}} \sum_{p \in S}\left(v_{\ell p} - \lambda_p\right).$$
\end{lemma}
\end{proof}

Proposition~\ref{thm:rounding_value} establishes the loss incurred from the rounding step, and follows from~\cite{SAP}. For the sake of completeness, we include the proof in Appendix~\ref{app:main-result-proofs}. 

\begin{proposition}\label{thm:rounding_value}
Let $\widetilde{ALG}$ denote the value of the solution immediately after the re-assignment step. Then, $\widetilde{ALG} \geq (1-\frac1e)OPT^{(\varepsilon)}$.
\end{proposition}

We next show that no additional loss is incurred in the aggregation step of our algorithm.
\begin{proposition}\label{prop:final-step-doesnt-matter}
The aggregation step maintains a feasible assignment of passengers to lines. Moreover, let $\widetilde{ALG}$ denote the value of the solution {\it before} the final aggregation step, and let $\{\widetilde{Y}_{\ell}\}$ and  $\{Y_{\ell}\}$ respectively denote the indicator variables corresponding to whether or not a line was opened, before and after the aggregation step; let $c(\widetilde{Y})$ and $c(Y)$ denote the costs of these respective solutions. Then, ${ALG} \geq \widetilde{ALG}$, and $c(Y) \leq c(\widetilde{Y})$. 
\end{proposition}

\begin{proof}
The fact that the objective weakly increases after the aggregation step follows from the fact that $\ell_1$ and $\ell_2$ share the same route, and $v_p(\cdot)$ is non-decreasing in the line frequency for all $p \in \mathcal{P}$. Moreover, $c(Y) \leq c(\widetilde{Y})$ follows from subadditivity of the cost function.

We now argue that a feasible assignment of passengers to lines is maintained after the aggregation step, i.e., that the bus capacity constraint is not violated for line $\ell' = (r, f_{\ell_1} + f_{\ell_2})$. Let $\left\{X_{\ell p}\right\}$ and $\{\widetilde{X}_{\ell p}\}$ be the indicator variables respectively denoting the assignment of passengers to lines, after and before the aggregation step. For all $e \in r$, we have:
\begin{align*}
    \sum_{p: e \in r_{\ell' p}}X_{\ell' p} \stackrel{(a)}{=} \sum_{p:e \in r_{\ell_1 p}} \widetilde{X}_{\ell_1 p} + \sum_{p:e \in r_{\ell_2 p}} \widetilde{X}_{\ell_2 p} \stackrel{(b)}{\leq} C(f_{\ell_1} + f_{\ell_2}),
\end{align*}
where~$(a)$ follows from the aggregation construction and~$(b)$ follows from the fact that the assignment of passengers to lines {\it before} the aggregation step was feasible by construction, for both $\ell_1$ and $\ell_2$. 
\end{proof}

To complete the proof of the theorem, we characterize the cost of the solution returned by Algorithm~\ref{alg:rounding}. We defer the proof of Proposition~\ref{prop:cost-characterization} to Appendix~\ref{app:main-result-proofs}.

\begin{proposition}\label{prop:cost-characterization}
The solution returned by Algorithm~\ref{alg:rounding} satisfies the budget constraint in expectation. Moreover, for all $\delta \in (0,1]$, the cost of the solution returned by Algorithm~\ref{alg:rounding} is at most $B(1-\varepsilon)(1+\delta)$ with probability at least $1 - e^{-\delta^2(1-\varepsilon)B/3c_{\max}}$.
\end{proposition}

The probabilistic budget guarantee follows from taking $\delta = \frac{\varepsilon}{1-\varepsilon}$.

\begin{corollary}\label{cor:cost}
The cost of the solution returned by Algorithm~\ref{alg:rounding} satisfies the budget constraint with probability at least $1-e^{-\varepsilon^2 B/3c_{\max}}$.
\end{corollary}

We complete the proof of Theorem~\ref{thm:main-thm} by putting together the facts established above.

\begin{proof}[Proof of Theorem~\ref{thm:main-thm}.]
Corollary~\ref{cor:cost} establishes the cost characterization.

For the approximation guarantee, putting together Theorem~\ref{thm:poly-time-separation} with Propositions~\ref{thm:rounding_value},~\ref{prop:final-step-doesnt-matter} and~\ref{prop:loss-from-epsilon}, we obtain that 
$$ALG \geq \left(1-\frac1e\right)OPT^{(\varepsilon)} \geq \left(1-\frac1e\right)(1-\varepsilon)OPT \geq \left(1-\frac1e-\varepsilon\right)OPT.$$
\end{proof}

\section{Numerical Experiments}
\label{sec:numerical-experiments}

{Finally, we complement our theoretical results by demonstrating the practical efficacy of our algorithm on: $(i)$ the Manhattan network, with real passenger data from for-hire vehicle ride requests, and $(ii)$ on a synthetic dataset based on a random network, designed to minimize any structural advantages. We present the former here, and defer our synthetic experiments to~\cref{app:synthetic_experiments}.}

We compare the solution returned by our algorithm to that of a state-of-the-art ILP solver, run on problem $(P)$ in Section~\ref{ssec:rlpp}. Note that the ILP solver cannot directly solve the configuration LP $\widehat{P}$, due to its exponential size, which is why instead feed it the natural formulation of the problem $(P)$. To emulate the real-time constraints on such a policy in practice, we run both our algorithm and the ILP solver under a strict time budget.

\subsection{Practical Implementation}
\label{procedure}

Although the theoretical analysis of our algorithm relies on using the ellipsoid method for solving the configuration LP, in practice, column generation is known to be more efficient (despite lacking poly-time guarantees){~\citep{desaulniers2006column}}. 
Thus, in our experiments we opt for column generation, where the generation of the new columns is done using our separation algorithm (Algorithm~\ref{alg:separation}).

Given an instance $I$ of \lpp, and parameters $\varepsilon \in (0,\frac12)$, $m \in \mathbb{N}$, we proceed as follows:

\begin{enumerate}
    \item Solve the configuration LP $\widehat{P}^{(\varepsilon)}$ in Algorithm~\ref{alg:rounding} via column generation. Return the current LP solution once the time budget has been exceeded.
    \item\label{step-2} Simulate the rounding through re-aggregation steps of Algorithm~\ref{alg:rounding} $m$ times. 
    \item Let $\mathcal{S}_B(I)$ denote the set of all budget-respecting solutions of the $m$ realized solutions;
    return the solution of maximum value in $\mathcal{S}_B(I)$.
\end{enumerate}

We note that this procedure retains our polynomial-time guarantees. Moreover, it benefits from the fact that Step~\ref{step-2} is easily parallelizable. In our experiments, we use $\varepsilon = 0.05$ and $m = 10^4$.

\subsection{Experimental setup and results}

To test the performance of our algorithm in a realistic setting, we develop a new dataset for modeling Mobility-on-Demand platforms, based on the Manhattan road network. We obtain the network from the publicly available OpenStreetMap (OSM) geographical data~\citep{boeing2017osmnx}.

\noindent\textbf{Line inputs.} We set the size of the candidate set of lines to be $L = 1,000$, and generate the candidate set based on the skeleton method proposed by~\citet{SILMAN1974201}, by iteratively choosing four nodes in the graph, uniformly at random, and connecting them via shortest path. 
We also set {$c_{\ell}$ to be proportional to the total travel time between the start and end nodes of line $\ell$.}
{We set the bus capacity $C = 30$, and assume that all bus routes operate at frequency 1. {Note that increasing the frequency of a line is equivalent to duplicating a route of frequency 1 in our algorithm. In our synthetic experiments (\cref{app:synthetic_experiments}) we observe that our algorithm's performance improves relative to the ILP solver as the size of the candidate set of lines increases. Thus, assuming frequency 1 lines only serves as a lower bound on our algorithm's performance on the real-world dataset.}

\noindent\textbf{Passenger inputs.} 
We use records of for-hire vehicle trips in Manhattan using the New York City Open Data platform, considering an hour's worth of trip requests between 5pm and 6pm on the first Tuesday of February, March and April 2018.
Our time windows have $9983$, $13851$, and $12301$ trip requests respectively.
{We note that the more commonly-used taxicab and rideshare datasets are unsuitable for our setting, as these datasets are heavily biased towards short trips (indeed, running our algorithm on this data results in most trips using the car-only option). In contrast, the for-hire trips are longer, and hence lead to significant savings from multi-modal trips.}

{
For each trip, instead of exact pickup and drop-off coordinates, the dataset provides only origin and destination `areas' (the over 4,000 nodes in the Manhattan network are divided into 69 areas). Given the area of an origin or destination, we sample a node in the area from the network uniformly at random. For each passenger $p \in \mathcal{P}$ and line $\ell \in \mathcal{L}$, we define the passenger-line value to be the difference between the time travelled by car when using $\ell$ and the duration of the direct car trip. Thus, our objective function is proportional to the total reduction in miles travelled by car in the system.}
{We moreover impose the constraint that a passenger-line value is only positive if the travel time induced for the passenger is no more than $\beta$ times the time of a direct trip by car, and set this detour factor $\beta = 3$.}

\begin{figure}\label{fig:manhattan_plan}
    \centering
    \includegraphics[scale=0.19]{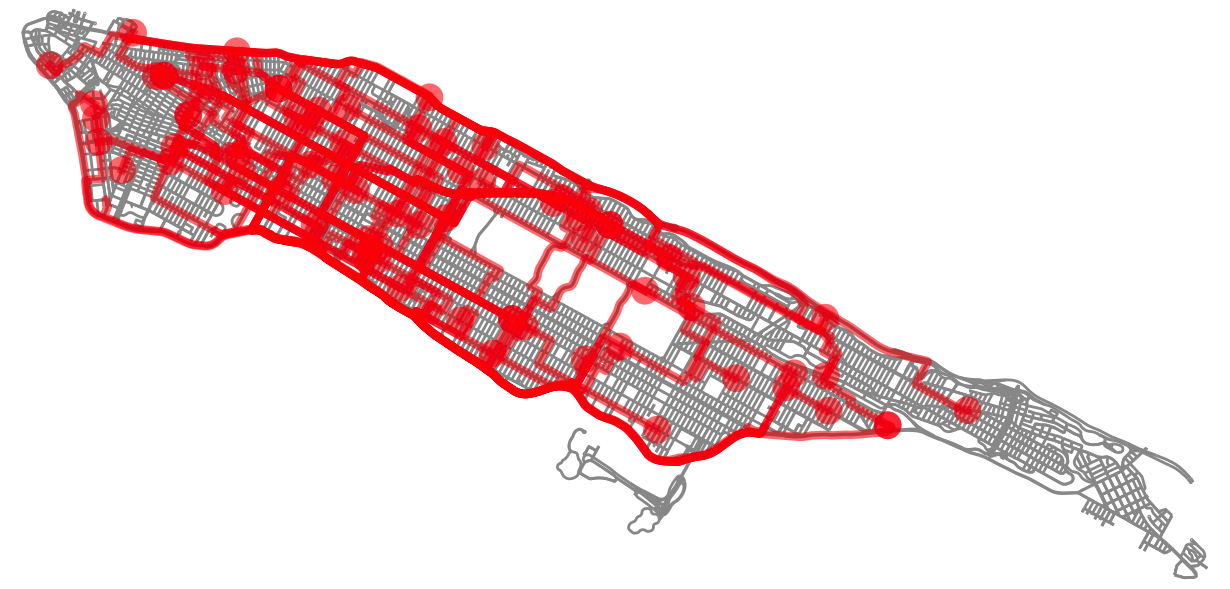}
    \caption{An example of line plan generated by our algorithm for the Manhattan network. We consider here the trip requests made on April 3, 2018 from 5pm to 6pm, with $B=3\cdot10^4$,  $L=10^3$ and $\beta =3$.}
\end{figure}

{We run the procedure for each of the three sets of requests, averaging the solutions returned by the procedure over these three instances. Let $ALG$ denote the corresponding empirical average. We also report $n_{\text{ILP}}$ and $n_{\text{ALG}}$, the number of lines respectively opened in the solutions returned by our algorithm and the ILP, and $\alpha$, the fraction of the outputs of the rounding process which were budget-respecting (out of the $m=10^4$ solutions of the rounding process)}. {Finally, we compute the {empirical average of the} multiplicative gap between the solution returned by our procedure and the value of the configuration LP $\widehat{P}^{(\varepsilon)}$ at the end of the allotted time. We use $\eta$ to denote this gap and note that, in cases where the configuration LP is not solved to optimality before rounding, $\eta$ may exceed 1.}

We report the results of our experiments in Table~\ref{tab:results_manhattan}. Our findings illustrate the practicality of our algorithm and relative inadequacy of the ILP for the task of real-time routing at scale.

\begin{table}[ht]
\centering
 \begin{tabular}{|c|cc|cccc|}
 \hline
 $B$ & ILP & $ALG$ & $n_{\text{ILP}}$ & $n_{\text{ALG}}$ & $\alpha$ & $\eta$ \\
  \hline \hline
  $10^4$  & \textbf{289,139} & $279,364$ 
  & $15$ & $16$ & $0.62$ & $0.87$ \\
  $2\cdot 10^4$ & $356,621$ & \textbf{509,586}  
  & $25$ & $27$ & $0.64$ & $0.94$ \\
  $3\cdot 10^4$  & ---  & \textbf{704,800} & --- & $36$ &  $0.65$ & $0.9$\\
  $5\cdot 10^4$ & --- &  \textbf{917,683} & --- & $60$ & $0.68$ & $0.89$\\
  $10^5$  & ---  & \textbf{1,140,700}  & --- & $106$ & $1$ & $1.12$ \\
  $2\cdot 10^5$ & \textbf{2,859,276} & $1,132,616$  & $242$ & $101$ & $1$ & $1.11$ \\
  \hline
  \end{tabular}
\caption{\emph{Numerical results for different budget values}: We set $L=10^3$, $\beta=3$. Bold values indicate the better solution for the corresponding value of $B$.  While the ILP outperforms our algorithm for the smallest and largest budgets, our algorithm consistently outperforms the ILP solver for more realistic intermediary budgets, where the ILP solver is often unable to return a solution within the allotted time.\\ 
The gap $\eta$ between the solution produced by our procedure and the value of the configuration LP at the end of the allotted time, is consistently above $87\%$, which is a significant improvement on the $0.95 \cdot \left(1-1/e\right)$ (i.e., $60\%$) theoretical guarantee. For larger budgets (i.e., between $10^5$ and $2\cdot 10^5$), the performance of our algorithm plateaus, as the column generation process requires more iterations to optimally solve the configuration LP. 
For the largest budget of $2\cdot 10^5$, the ILP is again able to get a solution by opening 242 lines (approximately a quarter of the candidate set). We conjecture that, with such a large budget, any set of lines is good enough, while more refined search is necessary to find the optimal lines for a more restricted budget. 
} \label{tab:results_manhattan}
\end{table}

\section{Conclusion}\label{sec:conclusion}

The integration of ride-hailing platforms' flexible demand-responsive services with the sustainability of mass transit systems is the next frontier in urban mobility. As ride-hailing platforms such as Uber and Lyft expand their range of services and look to adding high-capacity vehicles such as buses and shuttles to their fleets, they are faced with the following operational question: {\it Given a set of dynamically changing trip requests and a fleet of high-capacity vehicles, what is the optimal set of bus routes and corresponding frequencies with which to operate them?} 
In this work we provided a partial characterization of the hardness landscape of the \lineplanningproblem{} by proving that, unless the \platform{} has access to an existing candidate set of lines and passengers can only travel via one bus line (but are nevertheless allowed to transfer between bus and car services), the problem is hard to approximate within a constant factor. Under these assumptions, however, we developed a $1-\frac1e-\varepsilon$ approximation algorithm. We moreover demonstrated its efficacy in numerical experiments by showing that, when the \platform{} is constrained to short computation times (which is precisely the case if it wishes to be demand-responsive), then our algorithm outperforms exact methods on state-of-the-art ILP solvers.  

This paper lends itself to a number of natural directions for future work. From a theory perspective, though we showed that our algorithm can be modified with at most a constant-factor loss when the trip optimality assumption is relaxed, {existing approximation bounds for the interval scheduling problem are quite weak. An important area of investigation is whether we can leverage the additional structure of the \lineplanningproblem{} to strengthen the bounds of existing interval scheduling techniques.} 

\begin{acks}
This material is based upon work partially supported by the National Science Foundation under Grant No. CNS-1952011.
\end{acks}

\newpage
\bibliographystyle{ACM-Reference-Format}
\bibliography{main}


\newpage
\appendix

\section{Extensions}\label{ssec:extensions}

\subsection{Relaxing trip optimality.}\label{ssec:relaxing-trip-optimality} {In this section, we describe how our algorithm and analysis can be modified if the trip-optimality assumption (Assumption 2) is relaxed. Specifically, we no longer assume that passengers must use the trip option which maximizes their value along that line; the \platform{} must now consider all possible ways in which a passenger can join each line. We refer to this variant of the problem as the Generalized \lineplanningproblem{} (\textsc{GRlpp}).

Given line $\ell \in \mathcal{L}$, we define a {\it sub-route} of $r_{\ell}$ to be any set of consecutive edges of $r_{\ell}$. Let $n_{\ell}$ be the size of the set of all sub-routes of $r_{\ell}$. We index sub-routes of $r_{\ell}$ as $r_{\ell}^{(i)}$ for $i \in [n_{\ell}]$. Let $v_{\ell p}^{(i)}$ denote the value associated with passenger $p$ traveling along sub-route $r_{\ell}^{(i)}$ of $r_{\ell}$. Passenger $p$ can be assigned to any sub-route $r_{\ell}^{(i)}$ for which $v_{\ell p}^{(i)} > 0$. 

We first define the notion of {\it trip-optimality gap}.

\begin{definition}[Trip-optimality gap]
The {\it trip-optimality gap} $\gamma$ characterizes the worst-case multiplicative gap between the optimal values of \lpp{} and \textsc{GRlpp}. Formally, let $\mathcal{I}$ denote the set of all instances for the Generalized \lineplanningproblem. For $I \in \mathcal{I}$, $OPT(I)$ and $\widehat{OPT}(I)$ respectively denote the value of the optimal solution to \lpp{} and \textsc{GLpp}.  $$ \gamma = \sup_{I \in \mathcal{I}} \frac{\widehat{OPT}(I)}{OPT(I)}.$$
\end{definition}

\begin{proposition}\label{prop:trip-optimality-gap}
The \lineplanningproblem{} has unbounded trip-optimality gap.
\end{proposition}

\begin{proof}
Consider the setting where $|E| = |\mathcal{P}| = n-1$, $C = 1$, and $B$ is such that only one line $\ell$ at frequency 1 can be opened. Let $r_{\ell}^{(n)}$ denote the sub-route which uses all $n-1$ edges of $G$, and suppose $v_{\ell p}^{(n)} = 1$ for all $p \in \mathcal{P}$. Let $r_{\ell}^{(e)}$ denote the sub-route of $r_{\ell}$ which uses a single edge $e$, and suppose $v_{\ell p}^{(e)} = 1/2$ for all $e \in E, p \in \mathcal{P}$. Then, under the trip optimality assumption, the \lineplanningproblem{} has optimal value 1 (since all passengers must be served on {$r_{\ell}^{(n)}$} but $C = 1$). When this assumption is relaxed, however, the optimal value is {\it at least} $\frac{n-1}{2}$, achieved by having each passenger travel along a different edge.
\end{proof}

An unbounded trip-optimality gap would lead one to think that the more general, relaxed problem would require a fundamentally different approach from that of our algorithm. We however demonstrate the flexibility of our approach by proving that our algorithm can easily be modified for this setting, with {\it at most} a constant-factor loss.

We first introduce the following notation. Let ${S}_i$ denote a feasible assignment of passengers to {\it sub-route} $r_{\ell}^{(i)}$ of line $\ell$, for $i \in [n_{\ell}]$. Now, ${S} = ({S}_1,\ldots,{S}_{n_{\ell}})$ denotes a feasible assignment of passengers to {\it line} $\ell$. For ${S}$ to be feasible, $\{{S}_i\}$ must be disjoint subsets of $\mathcal{P}$ (i.e., a passenger can only be matched to one trip option), and the number of passengers using edge $e$ of $r_{\ell}$ must not exceed the capacity of the line. {Let ${\mathcal{I}_{\ell}}$ denote the set of feasible assignments of passengers to $\ell$.} For ease of notation, we use $p\in{S}$ if there exists $i \in [n_{\ell}]$ such that $p \in {S_i}$. 

We can still define an exponential-size configuration ILP for the Generalized \lineplanningproblem{}:

\begin{align}
{\widehat{P}:=} \qquad \max_{\left\{X_{\ell {S}}\right\}} \qquad
&  \sum_{\ell\in \mathcal{L}} \sum_{\substack{{S}\in {\mathcal{I}_{\ell}}}} \sum_{i \in [n_{\ell}]} \sum_{p \in {S}_i} v_{\ell p}^{(i)} X_{\ell {S}}  \notag\\
 \text{s.t.}\qquad & \sum_{\ell\in \mathcal{L}} c_\ell\left( \sum_{{S}\in \mathcal{I}_\ell} X_{\ell {S}}\right) \leq B \label{budget-config-lp}\\
 &    \sum_{{S}\in \mathcal{I}_\ell} X_{\ell{S}} \leq 1\qquad\qquad \forall \, \ell \in \mathcal{L} \label{one-set-per-line} \\
 &\sum_{\ell\in \mathcal{L}} \sum_{\substack{{S} \in {\mathcal{I}_{\ell}}:\\p\in {S}}} X_{\ell {S}} \leq 1 \qquad \forall \, p \in \mathcal{P} \label{cust-to-one-set} \\
&  X_{\ell {S}} \in \{0,1\} \qquad\qquad \forall \, \ell\in \mathcal{L}, {S} \in {\mathcal{I}_{\ell}}\notag
\end{align}

We can apply Algorithm~\ref{alg:rounding} to this problem. As before, however, we require a subroutine which (approximately) solves $\widehat{P}^{(\varepsilon)}$, the auxiliary configuration LP. 

In Section~\ref{sec:main-result} we showed that we can solve $\widehat{P}^{(\varepsilon)}$ by applying the ellipsoid method to the dual problem $\widehat{D}^{(\varepsilon)}$, { assuming access to a polynomial-time separation oracle}. To this end, we showed that an exact polynomial-time separation algorithm was within reach due to additional structure induced by trip optimality. 

We adopt a similar approach for the Generalized \lineplanningproblem. Consider the dual corresponding to $\widehat{P}^{(\varepsilon)}$, which we denote as before $\widehat{D}^{(\varepsilon)}$:

\begin{align*}
\widehat{D}^{(\varepsilon)} := \qquad \min_{\substack{\{q_{\ell}\}, \{\lambda_p\}, \alpha}}
 \qquad & \sum_{\ell \in \mathcal{L}} q_\ell + \sum_{p \in \mathcal{P}} \lambda_p +B(1-\varepsilon)\alpha\\
 \text{s.t.} \qquad  &   q_\ell + \alpha c_\ell \geq \sum_{i \in [n_{\ell}]} \sum_{p \in {S}_i} \left(v_{\ell p}^{(i)} - \lambda_p\right) \qquad \forall \, \ell \in \mathcal{L}, {S} \in {\mathcal{I}_{\ell}} \\
 &    q_\ell \geq 0 \quad \forall \, \ell \in \mathcal{L},\quad \lambda_p \geq 0 \quad \forall \, p \in \mathcal{P}, \quad \alpha \geq 0
\end{align*}

Recall, for fixed $\ell$, given $q_{\ell}, c_{\ell}, \{\lambda_{p}\}$, a separation algorithm for $\widehat{D}^{(\varepsilon)}$ either certifies that $q_\ell + \alpha c_\ell \geq \sum_{i \in [n_{\ell}]} \sum_{p \in {S}_i} \left(v_{\ell p}^{(i)} - \lambda_p\right) \, \forall {S} \in {\mathcal{I}_{\ell}}$, or returns ${S}$ such that this constraint is violated. {This can be done by solving the} following combinatorial optimization problem:

$$\max_{{S} \in {\mathcal{I}_{\ell}}}\sum_{i \in [n_{\ell}]}\sum_{p \in {S}_i} \left(v_{\ell p}^{(i)} - \lambda_p\right).$$

The following lemma follows from~\citet{SAP}.

\begin{lemma}[\citep{SAP}]\label{prop:beta-approx}
A $\beta$-approximate separation algorithm for $\widehat{D}^{(\varepsilon)}$ implies a $\beta$-approximation for $\widehat{D}^{(\varepsilon)}$.
\end{lemma}

Thus, given a constant-factor approximation for the separation algorithm, a constant-factor approximation for the Generalized \lineplanningproblem{} follows. 

\begin{corollary}\label{prop:glpp}
Let $\mathcal{A}$ be a $\beta$-approximate separation algorithm for $\widehat{D}^{(\varepsilon)}$. Then, using $\mathcal{A}$ as a sub-routine to Algorithm~\ref{alg:rounding} guarantees a $\left((1-\frac1e)\beta-\varepsilon\right)$-approximation for the Generalized \lineplanningproblem{} that is budget-respecting with probability at least $1-e^{-\varepsilon^2 B/3c_{\max}}$.
\end{corollary}

It suffices to show that such a constant-factor approximation exists. To see this, we show that the problem of finding a separation algorithm for $\widehat{D}^{(\varepsilon)}$ reduces to an instance of the Weighted Job Interval Selection problem (\textsc{Wjis}), for which a $\tfrac18$-approximation exists~\citep{Erlebach2003}. Establishing this analogy then completes the argument that we can use our algorithm to obtain a constant-factor approximation for the Generalized \lineplanningproblem{}.

\begin{definition}[Weighted Job Interval Selection Problem] Consider a set of $n$ jobs, $m$ machines, and a set of intervals $\mathcal{I}$ of the real line. Each job $j$ is defined by a set of feasible intervals $I_j \in \mathcal{I}$ in which it can be processed, as well as associated weights $\{w_{ij}\}_{i\in I_j}$. The goal is to select a subset of the intervals of maximum weight such that: $(i)$ at most one interval is selected for each job, and $(ii)$ at any point on the real line, no more than $m$ jobs can be scheduled.
\end{definition}

The analogy between the Generalized \lineplanningproblem{} and the Weighted Job Interval Selection Problem is as follows. Each line $\ell \in \mathcal{L}$ corresponds to the real line, and sub-route $r_{\ell}^{(i)}$ corresponds to an interval $i$ of the real line. Each passenger $p$ corresponds to a job $j$, and $v_{\ell p}^{(i)}-\lambda_p$ corresponds to the weight of processing job $j$ on interval $i$. The bus capacity $C$ is the number of machines. Thus, from any feasible solution to \textsc{Wjis} we can construct an assignment of passengers to sub-routes $\{r_{\ell}^{(i)}\}_{i \in [n_{\ell}]}$ such that each passenger is only assigned to one sub-route and the capacity $C$ of a bus on $r_{\ell}$ is nowhere exceeded. Such an assignment is thus feasible for line $\ell$, and any $\beta$-approximation for \textsc{Wjis} also gives us a $\beta$-approximate separation oracle for $\widehat{D}^{(\varepsilon)}$.} 

We briefly note that {$n_{\ell}$ is polynomial in $n$ since we've assumed that the maximum duration (weight) of a route is upper bounded by $D$, and the edge travel times are bounded below by a constant $\tau_{\min} > 0$.}
 Thus, since  Algorithm~\ref{alg:rounding} runs in polynomial time for the \lineplanningproblem{}, it also runs in polynomial time for the Generalized \lineplanningproblem.

\subsection{Travel times.}\label{ssec:travel-times}
We now show that abstracting away notions of travel and clock times is indeed without loss of generality, and that all results continue to hold for a more realistic, time-centric model. 

Let $T$ denote the length of the discrete time window during which the \platform{}  must serve the trip requests. A passenger is now defined by her source and destination nodes $s_p$ and $d_p$, as well as the time of her trip request $t_p$. Let $\mathcal{P}_T$ denote the set of all passengers. Clearly, $|\mathcal{P}_T| = |\mathcal{P}|$. In the same vein, a line is now defined by a route, a frequency, and a start time. Formally, the set of all possible lines the \platform{}  can operate is $\mathcal{L}_T = \left\{(r,f,t) \mid (r,f,t) \in \mathcal{R}\times\mathbb{N}\times [T]\right\}$. In this case, we have $|\mathcal{L}_T| = T|\mathcal{L}|$. Given the set of travel times $\{\tau_{ij}\}$, the \platform{}  can pre-compute the bus schedule induced by each line (e.g., if $(i,j) \in r_{\ell}$, and the bus leaves node $i$ at time $t$, then it reaches node $j$ at time $t + \tau_{ij}$). With slight abuse of notation, let $t_{\ell i}$ denote the time at which line $\ell$ reaches node $i$. Then, the only feasible trip options for passenger $p$ via line $\ell$ are $\omega = (s_p, i, j, d_p)$ such that $t_p + \tau_{s_p,i}^\star \leq t_{\ell i}$, where $\tau_{s_p,i}^\star$ is the car travel time from $s_p$ to $i$ (i.e., the duration of the shortest path between the two nodes). Given the bus schedule $\{t_{\ell i}\}$ and the passenger set $\mathcal{P}_T$, the \platform{}  can then pre-compute the passenger-line values $\{v_{\ell p}\}$. The size of each input to Algorithm~\ref{alg:rounding} has increased at most by a constant factor $T$. Hence, our algorithm still runs in polynomial time under this time-sensitive construction.

\section{Omitted proofs}

\subsection{Limits of approximation for the \lineplanningproblem}\label{app:hardness-proofs}

\subsubsection{Necessity of a candidate set of lines}

\begin{proof}[Proof of Proposition~\ref{prop:helper}]
 Since there is an edge of cost $t$ between $r$ and $g_i$, by definition of a min-cost path, $t_{p_i}^\star \leq t$. Thus, $(1-\frac{\varepsilon}{t})t_{p_i}^\star \leq (1-\frac{\varepsilon}{t})t = t-\varepsilon$. This then implies that $v_{p_i}(\omega_{\ell, p_i}) = 0$ for $\omega_{\ell, p_i}$ such that $t^{\text{car}}(\omega_{\ell, p_i}) > t-\varepsilon$. 
\end{proof}

\begin{proof}[Proof of Proposition~\ref{prop:hardness-helper}]
Consider passenger $p_i$. {We first claim that it is without loss of generality to assume that direct travel by car is completed via edge $(r,g_i)$. This is due to the fact that the cost of each edge of $G$ is lower bounded by $\varepsilon > 0$, and the cost of $(j,g_i)$ is $t-\varepsilon$ for all $j \in S_i$. Thus, traveling from $r$ to $g_i$ via $j \in S_i$ costs {\it at least} $t$, which is exactly the cost of the trip which uses edge $(r,g_i)$.} 

The fact that a bus line cannot be routed via edge $(r,g_i)$ follows from the fact that the cost of $(r,g_i)$ is $t > D$, and as such is infeasible {by bus since the maximum cost of a bus line is $D$}.

{The fact that a bus line cannot be routed via edge $(j,g_i)$ follows from the fact that the cost of $(j,g_i)$ is $t-\varepsilon > D$ for all $j \in S_i$. Now, suppose that the passenger travels via line $\ell$, and let $j_{-i}$ denote the vertex at which $p_i$ leaves the line and begins her journey by car. If $j_{-i} \notin S_i$, then reaching $g_i$ by car must incur a cost of at least $t$ (a cost of at least $\varepsilon$ to reach a node $j'\in S_i$, then a cost $t-\varepsilon$ to reach $g_i$ from $j'$). Thus, the value for this trip option is 0 by Proposition~\ref{prop:helper}, and the passenger would opt for a direct travel by car via edge $(r,g_i)$.}

\end{proof}

\begin{proof}[Proof of Proposition~\ref{prop:collect-pos-val}]

{The proposition follows immediately from  Proposition~\ref{prop:hardness-helper}. The only feasible options for passenger $p_i$ which use a bus line and collect strictly positive value are those for which a node in $S_i$ can be reached by bus.}
\end{proof}

\subsubsection{Inefficacy of standard approximation techniques}

\begin{proof}[Proof of Proposition~\ref{prop:not-submodular}]

Let $S$ denote the set of lines opened under $\by$. With mild abuse of notation, we use $w(S)$ to denote the welfare induced by this set of lines.

Consider the setting with three passengers $p_1, p_2, p_3$,  $\mathcal{L} = \{\ell_1, \ell_2, \ell_3\}$, $f_{\ell} = 1 \, \forall \, \ell \in \mathcal{L}$ and $C = 1$. The value functions associated with each passenger are as follows:
\begin{align*}
    v_{\ell 1} = \begin{cases}
    1 &\quad \text{if } \ell = \ell_1 \\
    1 &\quad \text{if } \ell = \ell_2 \\
    0 &\quad \text{if } \ell = \ell_3
    \end{cases}, \qquad v_{\ell 2} = \begin{cases}
    1 &\quad \text{if } \ell = \ell_1 \\
    0 &\quad \text{if } \ell = \ell_2 \\
    1 &\quad \text{if } \ell = \ell_3
    \end{cases}, \qquad v_{\ell 3} = \begin{cases}
    1 &\quad \text{if } \ell = \ell_1 \\
    0 &\quad \text{if } \ell = \ell_2 \\
    0 &\quad \text{if } \ell = \ell_3.
    \end{cases}
\end{align*}
Passengers $p_1$ and $p_2$ use disjoint edges of $r_{\ell_1}$. Passenger $p_3$, on the other hand, uses the same edges of $r_{\ell_1}$ as $p_1$ and $p_2$. Thus, any feasible assignment of $p_3$ to $\ell_1$ requires $p_3$ to be its sole passenger.  

Let $S_1=\{\ell_1\}$. Then, $w(S_1) = 2$, achieved by assigning $p_1$ and $p_2$ to $\ell_1$. Moreover, $w(S_1\cup \{\ell_3\}) = 2$, by assigning $p_1$ and $p_2$ to $\ell_1$, or $p_1$ to $\ell_1$ and $p_2$ to $\ell_3$.  Now, let $S_2 = \{\ell_1, \ell_2\}$. Again, by assigning $p_1$ and $p_2$ to $\ell_1$, we obtain $w(S_2)=2$. Moreover, $w(S_2\cup \{\ell_3\}) = 3$, obtained by assigning $p_1$ to $\ell_2$, $p_2$ to $\ell_3$ and $p_3$ to $\ell_1$.

Since $S_1 \subset S_2$ and $w(S_1 \cup \{\ell_3\}) - w(S_1) < w(S_2 \cup \{\ell_3\}) - w(S_2)$, $w$ is not submodular. 
\end{proof}

\begin{proof}[Proof of Proposition~\ref{prop:integrality-gap}]
Consider passengers $p_1, p_2$ and lines $\ell_1, \ell_2$ such that $$v_{\ell_1,p_1} = v_{\ell_2,p_2} = 1 \quad ,  \quad v_{\ell_2,p_1} = v_{\ell_1,p_2} = 0$$
with $r_{\ell_1}$ and $r_{\ell_2}$ non-overlapping. Suppose moreover that $c_{\ell_1} = c_{\ell_2} = 1$, and $B = 2-\varepsilon$, for some $\varepsilon \in (0,1)$. 

Since the ILP can only open a single line, its optimal value is $OPT = 1$. An optimal solution to the LP relaxation of the ILP, on the other hand, is such that $y_{\ell_1}^\star = 1, y_{\ell_2}^\star = 1-\varepsilon$, and thus its optimal value is $\widehat{OPT} = 2-\varepsilon$. Taking $\varepsilon \to 0$ proves the claim.
\end{proof}

\subsection{Main result: a $1-\frac1e-\varepsilon$ approximation algorithm}\label{app:main-result-proofs}

\begin{proof}[Proof of Lemma~\ref{lem:correctness}]
Let $A \in \mathbb{R}^{|r_{\ell}| \times |\mathcal{P}|}$ denote the constraint matrix corresponding to~\eqref{single_line_subproblem}.  $A$ is such that $A_{e, p} = 1$ if passenger $p$ uses edge $e$, and 0 otherwise. Since a passenger exclusively uses {\it consecutive} edges of $r_{\ell}$, the columns of $A$ have the consecutive-ones property. Thus, $A$ is totally unimodular. Since $C$ and $f_{\ell}$ are integral by assumption, these two facts together imply that LP-SEP is integral.
\end{proof}

\begin{proof}[Proof of Proposition~\ref{thm:rounding_value}~\citep{SAP}]
Let $ALG(p)$ denote passenger $p$'s expected contribution to the objective in the solution returned by our algorithm. To prove the approximation guarantee, it suffices to show the following:

$$ALG(p) \geq \left(1-\frac1e\right)\sum_{\ell \in \mathcal{L}}\sum_{S \in \mathcal{I}_{\ell}:p\in S} {v_{\ell p}}X_{\ell S}^{(\varepsilon)} \quad \forall \, p \in \mathcal{P}$$
where $\left\{X_{\ell S}^{(\varepsilon)}\right\}$ is the solution to $\widehat{P}^{(\varepsilon)}$. Summing over all $p$ and using $\sum_p \sum_{\ell \in \mathcal{L}}\sum_{S \in \mathcal{I}_{\ell}:p\in S} {v_{\ell p}}X_{\ell S}^{(\varepsilon)} = OPT^{(\varepsilon)}$ completes the proof of the result.

For each passenger $p\in \mathcal{P}$, let $Y_{\ell p} = \sum_{S \in \mathcal{I}_{\ell}: p \in S}X_{\ell S}^{(\varepsilon)}$. Sort the lines for which $Y_{\ell p} > 0$ in decreasing order of $v_{\ell p}$. Let $\{\ell_1,\ell_2,\ldots,\ell_k\}$ denote these lines, with $v_{\ell_1, p} \geq v_{\ell_2, p} \geq \ldots v_{\ell_k, p}$.

After rounding and re-assignment (R\&R), passenger $p$ is assigned to line $\ell_1$ if any set containing $p$ is assigned to $\ell_1$. Thus, $p$ is assigned to $\ell_1$ with probability $Y_{\ell_1, p}$. If no set containing passenger $p$ is assigned to $\ell_1$ after R\&R, then we look to $\ell_2$. The probability that a set containing $p$ is assigned to $\ell_2$ after R\&R is $Y_{\ell_2, p}$. Thus, $p$ is assigned to $\ell_2$ with probability $(1-Y_{\ell_1, p})Y_{\ell_2,p}$. It follows that, for all $k' \leq k$, passenger $p$ is assigned to $k'$ with probability $\prod_{i=1}^{k'-1}(1-{Y_{\ell_i, p}})Y_{k'}$. Hence, we have
$$ALG(p) = \sum_{k'=1}^k v_{\ell_{k'} p} Y_{k'} \left( \prod_{i=1}^{k'-1}(1-Y_{\ell_i, p})\right).$$

Lemma~\ref{lem:main-techical-lemma} relates $ALG(p)$ to the contribution of passenger $p$ {\it before} rounding and re-assignment, $\sum_{k'=1}^k v_{\ell_{k'},p} Y_{\ell_{k'},p}$.

\begin{lemma}[\citep{SAP}]\label{lem:main-techical-lemma}
Suppose $Y_{\ell p} \geq 0$ for all $\ell \in \mathcal{L}$, $\sum_{\ell \in \mathcal{L}}Y_{\ell p} \leq 1$, and $v_{\ell_1,p}\geq v_{\ell_2, p} \geq \ldots \geq v_{\ell_k, p} \geq 0$. Then
$$\sum_{k'=1}^k {v_{\ell_{k'} p}Y_{k'}}\left(\prod_{i=1}^{k'-1}(1-Y_{\ell_i, p})\right) \geq \left(1-(1-\frac1L)^L \right)\sum_{k'=1}^k v_{\ell_{k'},p} Y_{\ell_{k'},p}.$$

\end{lemma}

Using the fact that $(1-(1-\frac{1}{L})^L)\geq 1-\frac1e$ for all $L \geq 1$ completes the proof of Proposition~\ref{thm:rounding_value}.
\end{proof}

\begin{proof}[Proof of Proposition~\ref{prop:cost-characterization}]
Let $\widetilde{Y}_{\ell}$ be the indicator variable denoting the event that line $\ell$ was opened {\it before} the re-assignment step, and let $Y_{\ell}$ denote the final line status, after the aggregation step. Let $c(Y)$ and $c(\widetilde{Y})$ denote the total costs associated with $Y$ and $\widetilde{Y}$, respectively. Between the re-assignment step and the aggregation step, the cost of the solution could only have decreased, since lines were potentially closed. Similarly, by Proposition~\ref{prop:final-step-doesnt-matter}, the cost of the solution could only have decreased after the aggregation step. Thus, we have $c(Y) \leq c(\widetilde{Y})$, and\begin{align*}
\mathbb{E}\left[c(Y)\right] \leq \mathbb{E}\left[c(\widetilde{Y})\right] = \sum_{\ell \in \mathcal{L}} c_{\ell}\mathbb{P}\left[\widetilde{Y}_{\ell} = 1\right] = \sum_{\ell \in \mathcal{L}} c_{\ell}\left(\sum_{S \in \mathcal{I}_{\ell}}X_{\ell S}^{(\varepsilon)}\right) \leq B(1-\varepsilon),
\end{align*}
where the second inequality is by feasibility of $X_{\ell S}^{(\varepsilon)}$. Thus, the budget constraint is satisfied in expectation.

We now prove the second part of the claim:
\begin{align}
    \mathbb{P}\left(\sum_{\ell\in \mathcal{L}} c_\ell \, Y_{\ell}\geq (1+\delta)(1-\varepsilon)B\right) &{\leq}  \mathbb{P}\left(\sum_{\ell\in \mathcal{L}} c_\ell \, \widetilde{Y}_{\ell}\geq (1+\delta)(1-\varepsilon)B\right) \notag \\
    &= \mathbb{P}\left(\sum_{\ell\in \mathcal{L}} \frac{c_\ell}{c_{\max}} \, \widetilde{Y}_{\ell}\geq (1+\delta)\frac{(1-\varepsilon)B}{c_{\max}}\right) \notag
    \\&{\leq} e^{-\delta^2(1-\varepsilon)B/3c_{\max}} \label{eq:hoeff}
\end{align}
where~\eqref{eq:hoeff} follows from an application of the Chernoff bound to the independent random variables $\left\{\frac{c_{\ell}}{c_{\max}}\widetilde{Y}_{\ell}\right\}_{\ell \in \mathcal{L}}$, and uses the fact that $\mathbb{E}\left[\sum_{\ell}\frac{c_{\ell}}{c_{\max}}\widetilde{Y}_{\ell}\right] \leq \frac{(1-\varepsilon)B}{c_{\max}}$ by feasibility of $\left\{X_{\ell S}^{(\varepsilon)}\right\}$. 
\end{proof}

\section{Additional numerical experiments on synthetic data}
\label{app:synthetic_experiments}
{To complement our real-world data experiments, we consider a synthetic dataset and show how the performance of our algorithm depends on the number of requests and the cardinality of the candidate set of lines, using the ILP as a benchmark}. 

Observe that our algorithm relies on the underlying road network solely through the candidate set of lines $\mathcal{L}$, the line costs $\{c_{\ell}\}$, and the passenger-line values $\{v_{\ell p}\}$. Thus, it suffices to directly generate these latter sets of inputs, rather than inheriting them from an underlying structured network.  We note that generating inputs in this manner, rather than running our algorithm on a synthetic network (e.g., a grid network), further underscores the strength and generalizability of our scheme, as its success is not tied to the geometry of any underlying graph.

\noindent\textbf{Line inputs.} We generate the candidate set of lines as follows. For each $\ell \in \mathcal{L}$, we associate $D_{\ell}$ edges, where $D_{\ell}\sim Unif\{5,50\}$. Moreover, let $c_{\ell} = 1 \, \forall \, \ell \in \mathcal{L}$. {This implies that a \platform{} with budget $B$ can open {at most} $B$ lines.} Let $\mathcal{F}$ denote the set of possible frequencies with which to operate each bus route. In our first set of experiments, we let $\mathcal{F} = \{1\}$. Doing so is without loss of generality since, by definition, bus routes operated at different frequencies are considered to be different lines. Thus, considering a larger set of frequencies is computationally equivalent to increasing the size of candidate set of lines (e.g., considering 1,000 lines with 2 different frequencies is equivalent to considering 2,000 lines with a single frequency). We set the bus capacity $C = 30$.

 \noindent\textbf{Passenger inputs.} For each passenger $p \in \mathcal{P}$ and line $\ell \in \mathcal{L}$, we let $r_{\ell p}$ be a random subset of contiguous edges of $r_{\ell}$. To model the fact that, in a realistic network, passengers would not be covered by all lines, we define random variable $Z_{\ell p}~\sim Ber(0.1)$ representing whether or not passenger $p$ is covered by line $\ell$. Given $Z_{\ell p}$, we define the passenger-line value as follows:
 $$v_{\ell p} = \begin{cases}
 Unif[0,1] &\quad \text{ if } Z_{\ell p} = 1 \\
 0 &\quad \text { otherwise.}
 \end{cases}
 $$

 \noindent\textbf{Performance metrics.} We investigate the performance of this practical procedure along three dimensions: $(i)$ the number of passengers $N$, $(ii)$ the size of the candidate set of lines $L$, and $(iii)$ the \platform's budget $B$. {For both the ILP and our algorithm, we set a strict time limit of 20 minutes, and compare the solutions returned by the two schemes at the end of the allotted time.} 

 Given an instance of line and passenger inputs, we run the procedure described in Section~\ref{procedure} $M = 500$ times for each combination of parameters $(L,N,B)$ (i.e., we find the maximum of the $m=10^4$ realized solutions $M=500$ times).

 We run the procedure for 5 randomly-generated instances of line and passenger inputs. Let $ALG$ denote the empirical average of the solution returned by the procedure. 
 As before, we compute $\eta$ the {empirical average of the} multiplicative gap between the solution returned by our procedure and the value of the configuration LP $\widehat{P}^{(\varepsilon)}$ at the end of the allotted time.

We report the results of our experiments in Table~\ref{tab:results}. 


Whereas in theory the ILP solver provides an upper bound on $ALG$, this does not necessarily hold in our numerical results. This is due to the fact that, for large-scale problems and under reduced time budgets (i.e., the real-time application we are interested in), the ILP solver cannot solve the problem to optimality, and as such the objective it achieves is not necessarily an upper bound on $ALG$ in practice.

\begin{table}[ht]
\centering
    \begin{tabular}{|ccc|ccc|}
    \hline
      $L$ & $N$ & $B$   & ILP &   $ALG$ & $\eta$ \\
      \hline \hline
      $1,000$ & $5,000$ & 20 &  {2363}   & 2173 & 0.81 \\
      $5,000$ & $5,000$ & 20  &  2098  &  \textbf{2171} & 0.81 \\
      $7,000$ & $5,000$ & 20  & 807  & \textbf{2173} & 0.80\\
      $10,000$ & $5,000$ & 20  &  ---  &  \textbf{2174} &  0.81 \\
     \hline
     $5,000$ & $5,000$ & 20  & 2098  & \textbf{2171} & 0.81  \\
     $5,000$ &  $10,000$ & 20  & ---  &\textbf{3498} & 0.84 \\
     $5,000$ & $15,000$ & 20  & ---  & \textbf{4445} & 0.88  \\
     \hline
      \end{tabular}
      \qquad
      \begin{tabular}{|ccc|ccc|}
      \hline
      $L$ & $N$ & $B$   & ILP &   $ALG$ & $\eta$ \\
      \hline \hline
      $1,000$ & $5,000$ & 40  & 3671  & 2744 & 0.88 \\
      $5,000$ & $5,000$ & 40 &  3686 &  2743 & 0.88\\
      $7,000$ & $5,000$ & 40  & 2750 & \textbf{2754} & 0.88\\
      $10,000$ & $5,000$ & 40 & ---  & \textbf{2748}  & 0.88\\
     \hline
     $5,000$ & $5,000$ & 40  &  3686 &  2743 & 0.88  \\
     $5,000$ & $10,000$ & 40  & --- & \textbf{4949} & 0.85   \\
     $5,000$ & $15,000$ & 40  &  --- & \textbf{6691} & 0.82 \\
       \hline
\end{tabular}
\caption{Numerical results for budgets $B \in \{20,40\}$. Bolded values of $ALG$ indicate that our procedure outperforms the ILP benchmark for the corresponding $L, N, B$. While the ILP outperforms our algorithm {on smaller instances}, for larger values of $L$ and $N$, our algorithm consistently outperforms the ILP. {As the budget increases from 20 to 40, the ILP outperforms our algorithm for a larger set of values of $L$ and $N$; however, there still exists a threshold past which our algorithm outperforms the ILP. This difference is especially stark when $L$ and $N$ are both very large (we note that it is reasonable to expect $L$ and $N$ to grow with $B$): for these large-scale settings, the ILP is incapable of returning any feasible solution in the allotted time.} {Observe moreover that $\eta$, the gap between the solution produced by our procedure and the value of the configuration LP, is consistently above $80\%$, which is a significant improvement upon the $0.95 \cdot (1-\frac1e)$ (i.e., $60\%$) theoretical guarantee.}}\label{tab:results}
\end{table}


\end{document}